\documentclass[11pt,reqno]{amsart}
\usepackage{fullpage}

\usepackage{graphicx}
\usepackage{hyperref}
\usepackage{amsmath,amsopn,amssymb,amsfonts,stmaryrd}
\usepackage{verbatim}
\usepackage{amsthm}
\usepackage{mathtools}
\usepackage{color}
\usepackage{enumitem}
\usepackage[framemethod=TikZ]{mdframed}
\usepackage{bbm}
\usepackage{mathrsfs}
\usepackage{booktabs}
\usepackage{caption}
\usepackage{bm}
\usepackage{tensor}

\newcommand{\IZ}{{\mathbb Z}}
\newcommand{\IN}{{\mathbb N}}
\newcommand{\IH}{{\mathbb H}}

\newcommand{\sgn}{\mbox{sgn}}

\newcommand{\SL}{\mathrm{SL}}

\newcommand{\zz}{\mathfrak{z}}

\newcommand{\N}{\mathbb N}
\newcommand{\C}{\mathbb C}

\theoremstyle{plain}
\newtheorem{thm}{Theorem}[section]
\newtheorem{cor}[thm]{Corollary}
\newtheorem{lem}[thm]{Lemma}
\newtheorem{prop}[thm]{Proposition}

\theoremstyle{definition}

\newtheorem*{rem}{Remark}

\numberwithin{equation}{section}

\newcommand{\pmat}[1]{\left( \smallmatrix #1 \endsmallmatrix \right)}

\renewcommand{\sgn}{\textnormal{sgn}}

\def\lp{\left(}
\def\rp{\right)}
\def\lb{\left[}
\def\rb{\right]}

\def\a{\alpha}

\def\d{\delta}

\def\k{\kappa}

\def\w{\omega}

\def\z{\zeta}

\def\vth{\vartheta}
\def\ve{\varepsilon}
\def\e{\varepsilon}

\def\s{\sigma}

\def\t{\tau}

\def\a{\alpha}

\def\d{\delta}

\def\k{\kappa}

\def\w{\omega}

\def\z{\zeta}

\def\vth{\vartheta}
\def\ve{\varepsilon}
\def\e{\varepsilon}

\def\s{\sigma}

\def\t{\tau}

\def\DD{\Delta}
\def\LL{\Lambda}

\def\DD{\Delta}
\def\LL{\Lambda}

\def\del{  \partial}

\newcommand{\im}{{\rm Im}}

\renewcommand{\sgn}{{\rm sgn}}
\newcommand{\R}{\mathbb R}
\newcommand{\Z}{\mathbb Z}
\newcommand{\cS}{\mathcal{S}}
\newcommand{\vt}{\vartheta}

\def\wh{\widehat}

\setlist[itemize]{noitemsep, topsep=0pt}

\makeatletter
\newcommand{\vast}{\bBigg@{2}}
\newcommand{\Vast}{\bBigg@{5}}
\makeatother

\renewcommand{\pmod}[1]{\  \,  \left(  \mathrm{mod} \,  #1 \right)}

\newcommand{\rint}[2]{\tensor*[_{\ *\!\!\!}]{\int}{_{#1}^{#2}}}
\newcommand{\Res}{\operatorname{Res}}
\newcommand{\DDD}{\mathcal{D}}

\makeatletter
\newcommand{\nolisttopbreak}{\par\nobreak\@afterheading} 
\makeatother

\title{Integral Representations of Rank Two False Theta Functions and Their Modularity Properties}
\author{Kathrin Bringmann}
\address{University of Cologne, Department of Mathematics and Computer Science, Weyertal 86-90, 50931 Cologne, Germany}
\email{kbringma@math.uni-koeln.de}

\author{Jonas Kaszian}
\address{Max Planck Institute for Mathematics, Vivatsgasse 7, 53111 Bonn, Germany}
\email{jonask@mpim-bonn.mpg.de}

\author{Antun Milas}
\address{Department of Mathematics and Statistics, SUNY-Albany, Albany NY 12222, USA}
\email{amilas@albany.edu}

\author{Caner Nazaroglu}
\address{University of Cologne, Department of Mathematics and Computer Science, Weyertal 86-90, 50931 Cologne, Germany}
\email{cnazarog@math.uni-koeln.de}
\begin{document}

\begin{abstract}
False theta functions form a family of functions with intriguing modular properties and connections to mock modular forms. In this paper, we take the first step towards investigating modular transformations of higher rank false theta functions, following the example of higher depth mock modular forms.
In particular, we prove that under quite general conditions, a rank two false theta function is determined in terms of iterated, holomorphic, Eichler-type integrals. This provides a new method for examining their modular properties and we apply it in a variety of situations where rank two false theta functions arise. 
We first consider generic parafermion characters of vertex algebras of type $A_2$ and $B_2$. This requires a fairly non-trivial analysis of Fourier coefficients of meromorphic Jacobi forms of negative index, which is of independent interest. Then we discuss modularity of rank two false theta functions coming from superconformal Schur indices. Lastly, we analyze $\hat{Z}$-invariants of Gukov, Pei, Putrov, and Vafa for certain plumbing ${\tt H}$-graphs.
Along the way, our method clarifies previous results on depth two quantum modularity.
\end{abstract}

\maketitle

\section{Introduction and statement of results}

Modular forms and their variations provide a rich source of interaction between physics and mathematics. More recently, functions with more general forms of modular properties, such as mock modular forms, have gathered attention in both areas. In this paper, we focus on such a family of functions with generalized modularity properties called {\it false theta functions}. These are functions that are similar to ordinary theta functions on lattices with positive definite signature, except for certain extra sign functions, which prevent them from having the same simple modular properties as ordinary theta functions. For false theta functions over rank one lattices, one approach to understand them is by noting that they can be realized as holomorphic Eichler integrals of unary theta functions. This representation can be used to study the modular transformations of such functions and helps one understand why their limit to rational numbers yield quantum modular forms \cite{Zagier}. An alternative approach to modularity of false theta functions in 
\cite{CM1,CM2} is motivated by the concept of the $S$-matrix in conformal field theory. In this setup, false theta functions are  ``regularized" (defined  on $\mathbb{C} \times \mathbb{H}$, where $\mathbb{H}$ is the complex upper half-plane) and transform with integral kernels under the modular group. The $S$-kernel can be used to formulate a continuous version of the Verlinde formula \cite{CM1}. Yet another approach is to follow the example of mock modular forms and form a modular completion as done in \cite{BN}, where elliptic variables can also be naturally understood. The modular completion now depends on two complex variables in the upper half-plane $(\t,w) \in \mathbb{H} \times \mathbb{H}$, which transform in the same way under modular transformations,\footnote{A similar picture is obtained for mock modular forms by complexifying the complex conjugate of the modular variable $\tau$ so that we have a pair of complex variables $(\t,w)$ one living in the upper half-plane and one in the lower half-plane with both transforming in the same way under modular transformations.} and similar to mock modular forms, differentiating in  $w$ yields a modular form in $w$.

One of the main goals in this paper is to generalize the considerations from \cite{BN} to rank two false theta functions. As for rank one false theta functions, to study the modular transformations we follow the lead of higher depth mock modular forms, which were defined in unpublished work of Zagier and Zwegers and were recently developed through signature $(n,2)$ indefinite theta functions by \cite{ABMP}.\footnote{A notion that is similar to higher depth mock modular forms is that of polyharmonic Maass forms \cite{BDR, LR}. }
In particular, the double error functions introduced by \cite{ABMP} show how double products of sign functions can be replaced to give modular completions. In Lemma \ref{SignLemma} we give a particularly useful form to understand this fact in a shape suitable for our context. This result then suggests a notion of false theta functions at ``depth two", where we find a modular completion again depending on two complex variables $(\t,w) \in \mathbb{\IH} \times \mathbb{\IH} \setminus \{ \tau = w \}$ and where the derivative in $w$ leads to modular completions of the kind studied in \cite{BN}, which are at ``depth one".
More specifically, our result leads us to modular completions $\wh{f}(\t,w)$ which transform like modular forms under simultaneous modular transformations $(\tau, w) \mapsto ( \frac{a \tau + b}{c \tau + d}, \frac{a w+b}{c w+d} )$ for $\pmat{a & b \\ c & d} \in \mathrm{SL}_2 (\IZ)$ and reproduce the rank two false theta functions we are studying through the limit $\lim_{w \to \t + i \infty} \wh{f}(\t,w)$. Moreover, their  derivatives with respect to $w$ appear in the form
\begin{equation*}
\frac{\del \wh{f} (\tau,w)}{\del w} = \sum_j (i(w-\t))^{r_j} \, \wh{g}_j (\tau,w) \, h_j (w) ,
\end{equation*}
where $r_j \in \frac{\IZ}{2}$, $h_j$ is a weight $2+r_j$ modular form (with an appropriate multiplier system), and $\wh{g}_j (\tau,w)$ is a modular completion of the sort studied in \cite{BN}. This is a structure that closely resembles those of depth two mock modular forms. It would be interesting to elaborate on the details here and form an appropriate notion of ``higher depth false modular forms" by mirroring the structure of higher depth mock modular forms. We leave this problem as future work and restrict our attention to answering concrete modularity questions about rank two false theta functions arising in a variety of mathematical fields.

A rich source of false theta functions that is studied in this paper
is through the Fourier coefficients of meromorphic Jacobi forms with negative index or their multivariable generalizations \cite{BKMZ, BRZ}.\footnote{
Here and in the rest of this paper, whenever we say Fourier coefficients of (meromorphic) Jacobi forms, we mean Fourier coefficients with respect to the elliptic variables.
}
Such meromorphic Jacobi forms naturally arise in representation theory of affine Lie algebras and in conformal field theory. 
In vertex algebra theory, important examples of meromorphic Jacobi forms come from characters of irreducible modules for the simple affine vertex operator algebra $V_k(\mathfrak{g})$ at an admissible level $k$. At a boundary admissible level \cite{KW0}, these characters admit particularly elegant infinite product form. Modular properties of their Fourier coefficients are understood only for $\mathfrak{g}=\text{sl}_2$ and $V_{-\frac32}(\text{sl}_3)$. For the latter, the Fourier coefficients are essentially rank two false theta functions (see  \cite{BKMZ} for more details). On the very extreme, if the level is generic, the character of $V_k(\mathfrak{g})$ is given by
\begin{equation} \label{paraf}
{\rm ch}[V_k(\mathfrak{g})]({\bm \zeta};q)]=\frac{q^{-\frac{{\rm dim}(g)k}{24\left(k+h^\vee\right)}}}{(q;q)_\infty^n \prod_{\alpha \in \Delta_+} ({\bm \zeta}^{\alpha}q;q)_\infty \prod_{\alpha \in \Delta_+} ({\bm \zeta}^{-\alpha}q;q)_\infty},
\end{equation}
where $n$ is the rank of $\mathfrak{g}$, $h^\vee$ the dual Coxeter number, and as usual, $(a;q)_r := \prod_{j=0}^{r-1}(1-aq^j)$ for $r \in \mathbb{N}_0 \cup \{ \infty \}$. Moreover $\bm \zeta$ are variables parametrizing the set of positive roots $\Delta_+$  of $\mathfrak{g}$ 
and throughout this paper we use bold letters to denote vectors.
Although (\ref{paraf}) is not a Jacobi form, a slight modification in the Weyl denominator gives a genuine Jacobi form 
of negative index. The Fourier coefficients of (\ref{paraf}) are important because they are essentially 
characters for the  parafermion vertex algebra $N_k(\mathfrak{g})$ \cite{DL,DW,KP} (see also Section 5), whose character is given by
\begin{equation} \label{par_ct}
(q;q)_\infty^n {\rm CT}_{[\bm \zeta]} \left( {\rm ch}[V_k(\mathfrak{g})]({\bm \zeta};q)\right),
\end{equation}
where $\mathrm{CT}_{[\bm\z]} $ denotes the constant term in the expansion in $\zeta_j$.
The character can be expressed as linear combinations of coefficients of Jacobi forms. 
One of the goals of this paper is to investigate modular properties of (\ref{par_ct}) for types $A_2$ and $B_2$, which leads us to the following result.
\begin{thm}\label{thm:parafermion_modularity}
Characters of the parafermion vertex algebras of type $A_2$ and $B_2$ can be written as linear combinations of (quasi-)modular forms and  false theta functions of rank one and two. The rank two pieces in these decompositions can be written as iterated holomorphic Eichler-type integrals, which yields the modular transformation properties of these functions.
\end{thm}
Note that more precise versions of this result are given in Propositions \ref{prop_const_term_A2}, \ref{prop:Psi_integral_form}, \ref{prop:Psi_modular_properties}, \ref{prop_const_term_B2}, \ref{prop_phi_integral_form}, and \ref{prop:Phi_modular_properties}. 
Independent of modular properties, we expect that the analysis we make on the characters ${\rm ch}[V_k(\mathfrak{g})]$ in these two cases will also shed some light on the nature of coefficients of meromorphic, multivariable Jacobi forms of negative definite index. We furthermore hope that our techniques can be extended to study parafermionic characters at boundary admissible levels.

Meromorphic Jacobi forms closely related to characters of affine Lie algebras at boundary admissible levels also show up in the computation of the Schur index $\mathcal{I}(q)$ of 4d $\mathcal{N}=2$ superconformal field theories (SCFTs) \cite{Beem, BNi}. If refined by flavor symmetries, the Schur index is denoted by $\mathcal{I}(q,z_1,..,z_n)$. In this paper, we are only interested in the Schur index of some specific SCFTs, called Argyres--Douglas theories of type $(A_1,D_{2k+2})$, whose index with two flavors was first computed in \cite{BNi} (see also \cite{CS}) and later identified with certain vertex algebra  characters in \cite{Creutzig0}. 
In particular, for $k=1$ the index coincides with the character of the aforementioned vertex algebra $V_{-\frac32}(\rm{sl}_3)$.
Our second main result deals with modularity of Fourier coefficients of these indices; for a more precise statement see Section \ref{sec:Schur_Zhat}.
\begin{thm}\label{thm:schur_index_modularity}
The Fourier coefficients of the Schur indices of Argyres--Douglas theories of type $(A_1,D_{2k+2})$ are essentially rank two false theta functions. Moreover, the constant terms in these Fourier expansions can be expressed as double Eichler-type integrals.
\end{thm}

The third main result concerns the $\hat{Z}$-invariants, called homological blocks, of plumbed $3$-invariants introduced recently by  Gukov, Pei, Putrov, and Vafa \cite{Gukov} and further studied from several viewpoints in \cite{BMM, CCFGH, EGGJPS, GM, Gukov, Kuc, Park}. For Seifert homology spheres, it is well-known that they can be expressed as linear combinations of derivatives of unary false theta functions, whose modular properties are known. Further computations of $\hat{Z}$-invariants for certain
non-Seifert integral homology spheres were given in \cite{BMM}. Our next result is an integral representation of these invariants. Compared to \cite{BMM}, Theorem \ref{thm:z_invariant_representation} gives a more direct relationship between iterated Eichler integrals and $\hat{Z}$-invariants. 
\begin{thm}
\label{thm:z_invariant_representation} Let $M$ be a plumbed $3$-manifold obtained from a unimodular ${\tt H}$ graph as in \cite{BMM}. Then the $\hat{Z}$-invariant of $M$ has a representation of the shape
\begin{equation*}
\hat{Z}(\tau)=\int_{\tau}^{\tau+i\infty}  \int_{\tau}^{w_1} \frac{\Theta_1(w_1,w_2)}{ \sqrt{i (w_1-\tau)} \sqrt{i (w_2-\tau)}}d w_2 d w_1+ \Theta_2(\tau),
\end{equation*}
where $\Theta_1(w_1,w_2)$ is a linear combination of products of derivatives of unary theta functions in $w_1$ and $w_2$ and $\Theta_2(\tau)$ is a rank two 
theta function. Moreover, there is a completion of $\hat{Z}$ which transforms like a weight one modular form.\footnote{
In this paper, we employ ``hats" to denote modular completions as is common in the literature for mock modular forms. This should not be confused with the hat that appears in $\hat{Z}$ for homological blocks, which is also a standard notation in literature. }
\end{thm}

Importantly, Theorems \ref{thm:parafermion_modularity}, \ref{thm:schur_index_modularity}, and \ref{thm:z_invariant_representation} completely determine the modular properties of the functions under investigation. These results in turn pave the way for studying ``precision asymptotics" for the relevant functions within all the contexts stated above, i.e.,~characters of parafermionic algebras, supersymmetric Schur indices, and homological invariants of $3$-manifolds. In the case of classical modular forms, this is accomplished by studying Poincar\'e series and by using the Circle Method. The most classical example is the exact formula for the integer partition function found by Rademacher \cite{Rademacher}, whose convergent formula extended the asymptotic results of Hardy and Ramanujan significantly. In fact, such results are intimately related to the finite-dimensionality of the associated vector spaces of modular objects and this property forms the basis for many of the remarkable applications of modularity to different fields of mathematics.
The Circle Method has already been applied to a case involving rank one false theta functions in \cite{BN} and to one involving depth two mock modular forms in \cite{BN2}. It would be interesting to extend these results to the class of functions studied in this paper and explore the implications to the different fields considered here.

Finally, the outline of the paper is as follows:
In Section \ref{sec:preliminaries}, we gather several facts on certain classical modular forms, Jacobi theta functions, and a number of meromorphic Jacobi forms of two complex variables used in the paper. 
In Section \ref{sec:sign_lemma}, we prove Lemma \ref{SignLemma}, which is the main technical tool used to study rank two false theta functions as we demonstrate in the rest of the section.
Then in Section \ref{sec:meromorphic_Jacobi_decomp}, we collect several technical results used in studying Fourier coefficients of meromorphic Jacobi forms. In Section \ref{sec:A_2_parafermion}, we turn our attention to parafermionic characters of type $A_2$ and show that one can write them in terms of modular forms and a rank two false theta function. We then find the modular transformations of the rank two piece using tools from Section \ref{sec:sign_lemma}. In Section \ref{sec:B_2_parafermion}, we apply the same type of analysis on generic parafermionic characters of type $B_2$. In Section \ref{sec:Schur_Zhat}, we demonstrate how the tools used in this paper also applies to rank two false theta functions coming from superconformal Schur indices and $\hat{Z}$-invariants of $3$-manifolds. We conclude in Section \ref{sec:conclusion} with final remarks and comments on future prospects.

\section*{Acknowledgments}
The authors thank Sander Zwegers for fruitful discussions. The research of the first author is supported by the Alfried Krupp Prize for Young University Teachers of the Krupp foundation and has received funding from the European Research Council (ERC) under the European Union's Horizon 2020 research and innovation programme (grant agreement No. 101001179). The third author was partially supported by the NSF grant DMS-1601070 and a Simons Collaboration Grant for Mathematicians. The research of the fourth author is  supported by the SFB/TRR 191 ``Symplectic Structures in Geometry, Algebra and Dynamics", funded by the DFG (Projektnummer 281071066 TRR 191). Finally, we thank the referees for providing useful comments.

\section{Preliminaries}\label{sec:preliminaries}

We start by recalling several functions which we require in this paper. Firstly, let 
\begin{equation*}
\eta(\tau):=q^{\frac{1}{24}}\prod_{n=1}^\infty\left(1-q^n\right)
\end{equation*}
be \emph{Dedekind's $\eta$-function}, where $q:=e^{2\pi i \tau}$. It satisfies the modular transformations
\begin{equation*}
\eta(\tau+1)=e^{\frac{\pi i}{12}}\eta(\tau),\qquad
\eta\left(-\frac 1\tau\right)=\sqrt{-i\tau}\eta(\tau).
\end{equation*}
Note that these two transformations imply that for $M = \pmat{a & b \\ c & d} \in \SL_2 (\IZ)$ we have
\begin{equation*}
\eta \lp \frac{a \t+b}{c \t+d} \rp = \nu_\eta (M)  (c \t + d)^{\frac{1}{2}} \eta (\t),
\end{equation*}
where $\nu_\eta$ denotes the multiplier system for the $\eta$-function. We furthermore use the identity
\begin{equation*}\label{E3}
\eta(\tau)^3=\sum_{n\in \IZ} (-1)^n \lp n+\frac{1}{2} \rp q^{\frac{1}{2} \lp n+\frac{1}{2} \rp^2}.
\end{equation*}

We also require the \emph{Jacobi theta function} defined by ($\zeta:=e^{2\pi i z}$)
\begin{equation*}
\vartheta(z;\tau):=\sum_{n \in \Z+\frac{1}{2}}e^{\pi i n}q^{n^2}\zeta^n.
\end{equation*}
By the Jacobi triple product formula, we have the product expansion \begin{equation}\label{JT}
\vt(z;\tau)=-iq^{\frac 18}\zeta^{-\frac 12}(q;q)_\infty(\zeta;q)_\infty\left(\zeta^{-1}q;q\right)_\infty.
\end{equation}
The Jacobi theta function transforms like a Jacobi form of weight and index $\frac 12$: 
\begin{align}
\label{thetamod} \vt(z;\tau+1)&=e^{\frac{\pi i}{4}}\vt(z;\tau),\qquad
\vt\left(\frac z\tau;-\frac 1\tau \right) =-i\sqrt{-i\tau}e^{\frac{\pi i z^2}{\tau}}\vt(z;\tau), \\ \label{JE}
\vt(z+1;\tau)&=-\vt(z;\tau), \qquad \vt(z+\tau;\tau)=-q^{-\frac{1}{2}} \zeta^{-1}\vt(z;\tau).
\end{align}
\newline 
Moreover, we have
\begin{equation}\label{diffT}
\left[\frac{\partial}{\partial z}\vt(z;\tau)
\right]_{z=0}=-2\pi\eta(\tau)^3.
\end{equation}

We also need the unary theta functions
\begin{equation*}
\vth_{m, r} (z;\t) := \begin{cases} \sum\limits_{n \in \IZ + \frac{r}{2m}} q^{m n^2} \zeta^{2 m n} & \text{ if } m \in \Z,\\
\sum\limits_{n \in \IZ + \frac{r}{2m}+\frac 12} (-1)^{n - \frac{r+m}{2m}} q^{m n^2} \zeta^{2 m n} & \text{ if } m \in \Z +\frac{1}{2} . \end{cases}
\end{equation*}
They satisfy the following elliptic and modular transformations.
\begin{lem} \label{lem:theta_onedim_modular_transf}
\hfill \nolisttopbreak
\begin{enumerate}[label = \rm(\arabic*), itemindent=-0.7cm]
\item For $m \in \IZ$ and $r  \in \IZ / 2 m \IZ$, we have:
\begin{align*}
\vth_{m, r} (z; \t+1) &= e^{\frac{\pi i r^2}{2m}} \vth_{m, r} (z; \t),\\
\vth_{m, r} \lp \frac{z}{\t}; -\frac{1}{\t} \rp &= e^{ \frac{2\pi i m z^2}{\t}} 
\frac{\sqrt{-i\t}}{\sqrt{2m}} \sum_{\ell \pmod{2m}} e^{-\frac{\pi i r \ell}{m}} \vth_{m, \ell} (z; \t) .
\end{align*}

\item For $m \in \IZ+\frac{1}{2}$ and $r  \in \IZ / 2 m \IZ$, we have:
\begin{align*}
\vth_{m, r} (z; \t+1) &= e^{\frac{\pi i (r+m)^2}{2m}} \vth_{m, r} (z; \t),\\
\vth_{m, r} \lp \frac{z}{\t}; -\frac{1}{\t} \rp &=  e^{ \frac{2\pi i m z^2}{\t}} 
\frac{e^{-\pi i m}\sqrt{-i\t}}{\sqrt{2m}} 
\sum_{\ell \pmod{2m}} (-1)^{r+\ell} e^{-\frac{\pi i r \ell}{m}} \vth_{m, \ell} (z; \t) .
\end{align*}
\end{enumerate}
\end{lem}
We denote the derivatives of $\vth_{m, r}(z;\t)$ with respect to $z$ as:
\begin{equation*}
\vth_{m, r}^{[k]} (\t) := \left[ \lp \frac{1}{4 \pi i m} \frac{\del}{\del z} \rp^k \vth_{m, r} (z;\t) \right]_{z=0}.
\end{equation*}
Note that we drop the superscript if $k=0$.

Another function we use is the quasimodular Eisenstein series

\begin{equation*}
E_2 (\t) := 1 - 24 \sum_{n=1}^\infty \sum_{d|n}d q^{n},
\end{equation*}
which satisfies the (quasi)modular transformations
\begin{equation*}
E_2 (\t+1) = E_2 (\t), \qquad
E_2 \lp -\frac{1}{\t} \rp = \t^2 E_2 (\t) + \frac{6 \t}{\pi i}  .
\end{equation*}
This function is used in the definition of the {\it Ramanujan-Serre derivative},
\begin{equation*}
\DDD_k := \frac{1}{2\pi i} \frac{\del}{\del \t} -\frac{k}{12} E_2 (\t),
\end{equation*}
which maps modular forms of weight $k$ to modular forms of weight $k+2$.

Finally, in Sections \ref{sec:A_2_parafermion} and \ref{sec:B_2_parafermion}, we analyze Fourier coefficients of two multivariable meromorphic Jacobi forms defined as follows:
\begin{equation}\label{defineT}
T_{A} (\bm z;\tau) := 
\frac{1}{\vartheta(z_1;\tau)\vartheta(z_2;\tau)\vartheta(z_1+z_2;\tau)}
, \ \ \ \ T_{B}(\bm z;\tau):=\frac{T_A(\bm z;\tau)}{\vartheta(2z_1+z_2;\tau)}. 
\end{equation}
Here we recall that a Jacobi form $f:\C^N\times \mathbb{H}\to \C$ of weight $k\in\frac12\Z$ and matrix index 
$M\in \nolinebreak\frac14\Z^{N\times N}$ satisfies the following transformation laws (with multipliers $\nu_1, \nu_2$):
\begin{enumerate}[leftmargin=*]
\item For $\left(\begin{smallmatrix} a & b\\c & d \end{smallmatrix}\right)\in \SL_2(\Z)$ we have
\begin{align*}
f\left(\frac{\bm{z}}{c\tau+d};\frac{a\tau+b}{c\tau+d}\right) = \nu_1\pmat{a & b \\ c& d} \left(c\tau+d\right)^k e^{\frac{2\pi i c}{c\tau +d} \bm{z}^T M \bm{z}} f(\bm{z};\tau).
\end{align*}
\item For $(\bm{m}, \bm{\ell})\in\Z^N\times \Z^N$ we have
\begin{align*}
f\left(\bm{z}+\bm{m}\tau +\bm{\ell};\tau \right) = \nu_2(\bm{m},\bm{\ell}) q^{-\bm{m}^T M \bm{m}} e^{-4\pi i \bm{m}^T M \bm{z}} f(\bm{z};\tau).
\end{align*}
\end{enumerate}
From \eqref{thetamod} and \eqref{JE} we easily see that 
$T_A$ and $T_B$ transform like Jacobi forms with weights $-\frac{3}{2}$ and $-2$, and matrix indices $-\frac{1}{2}\left(\begin{smallmatrix} 2 & 1 \\  1 & 2 \end{smallmatrix}\right)$ and $-\frac{1}{2}\left(\begin{smallmatrix} 6 & 3  \\ 3 & 3 \end{smallmatrix}\right)$, respectively (with some multipliers). We also consider in Section \ref{sec:Schur_Zhat} for $k \in \mathbb{N}$, 
\begin{equation*} \label{schur}
\mathbb{T}_k(\bm z;\tau):=\frac{\vartheta(z_1; (k+1) \tau)\vartheta(z_2; (k+1) \tau)\vartheta(z_1+z_2; (k+1) \tau)}{\vartheta\left(z_1;  \tau\right)\vartheta\left(z_2;  \frac{k+1}{2} \tau\right)\vartheta \left(z_1+z_2; \frac{k+1}{2}  \tau\right)}.
\end{equation*}
The function $\mathbb{T}_k ((k+1) \bm{z};\tau)$ with rescaled elliptic variables is a Jacobi form of weight zero and matrix index $-\frac{k+1}{2}\left(\begin{smallmatrix} k+1 & 1 \\ 1 & 2 \end{smallmatrix}\right)$.

\section{Products of Sign Functions and Iterated Integrals}\label{sec:sign_lemma}

A key technical result in this paper is the following lemma which allows one to write products of sign functions in terms of iterated integrals. This lemma essentially follows from Proposition 3.8 of \cite{ABMP}, which gives an expression that allows efficient numeric evaluation of double error functions developed there. These double error functions play a fundamental role in understanding modular properties of indefinite theta functions for lattices of signature $(n,2)$. The double error functions become signs towards infinity and this is what we express in the next lemma. It is further processed and cast into a form from which the modular properties of false theta functions are manifest.

\begin{lem}\label{SignLemma}
For $\ell_1,\ell_2 \in \mathbb R$, $\kappa\in\R$, with $\lp \ell_1,\ell_2+ \kappa\ell_1 \rp \neq (0,0)$, we have 
\begin{align*}
&\textnormal{sgn}(\ell_1)\textnormal{sgn}(\ell_2+\kappa \ell_1)  q^{\frac{\ell_1^2}{2} + \frac{\ell_2^2}{2}} =\int_{\tau}^{\tau+i\infty}  \frac{\ell_1 e^{\pi i  \ell_1^2 w_1}}{\sqrt{i (w_1-\tau)}} 
\int_{\tau}^{w_1}   \frac{\ell_2 e^{\pi i \ell_2^2 w_2}}{\sqrt{i (w_2-\tau)}} d w_2 d w_1 \\
&\hspace{6em} \qquad +
\int_{\tau}^{\tau+i\infty }  \frac{m_1 e^{\pi i m_1^2 w_1}}{\sqrt{i (w_1-\tau)}} 
\int_{\tau}^{w_1}  \frac{m_2 e^{\pi i  m_2^2 w_2}}{\sqrt{i (w_2-\tau)}} d w_2 d w_1 + 
\frac{2}{\pi} \arctan(\kappa)  q^{\frac{\ell_1^2}{2} + \frac{\ell_2^2}{2}} ,
\end{align*}
where $\sgn (x) := \frac{x}{|x|}$ for $x \neq 0$, $\sgn (0) := 0$, $m_1 := \frac{\ell_2 + \kappa \ell_1}{\sqrt{1+\kappa^2}}$, and $m_2 := \frac{\ell_1 - \kappa \ell_2}{\sqrt{1+\kappa^2}}$. 
\end{lem}
\begin{rem}
We use $\t+i\infty$ in the upper limits of these integrals to indicate that all
such integrals are taken along the vertical path from $\t$ to $i \infty$ and we use the principal value of the square root. 
\end{rem}
\begin{proof}[Proof of Lemma \ref{SignLemma}]

We first assume that both $\ell_1, \ell_2+\kappa \ell_1 \neq 0$.
Shifting $w_j \mapsto i w_j + \tau$ the first term on the right-hand side of the lemma equals 
\begin{equation}\label{firstterm}
-\ell_1 \ell_2 q^{\frac{\ell_1^2}{2} + \frac{\ell_2^2}{2}}  \int_0^{\infty} \frac{e^{- \pi \ell_1^2 w_1}}{\sqrt{-w_1}} \int_0^{w_1} \frac{e^{- \pi \ell_2^2 w_2}}{\sqrt{-w_2}} d w_2 d w_1.
\end{equation}
On the path of integration, we have $\sqrt{-w_j} = i \sqrt{w_j}$. Changing $w_j\mapsto w_j^2$, equation \eqref{firstterm} thus equals 
\begin{equation*}
4\ell_1\ell_2 q^{\frac{\ell_1^2}{2} + \frac{\ell_2^2}{2}}  \int_0^\infty e^{-\pi \ell_1^2w_1^2} \int_0^{w_1} e^{-\pi \ell_2^2w_2^2} d w_2d w_1.
\end{equation*}
We then employ the following integral identity, which is straightforward to verify
\begin{equation*}
4\ell_1\ell_2\int_0^\infty e^{-\pi \ell_1^2w_1^2} \int_0^{w_1} e^{-\pi \ell_2^2w_2^2}d w_2d w_1= \frac{2}{\pi} \arctan \lp   \frac {\ell_2}{\ell_1}\rp.
\end{equation*}
Using that $m_1^2+m_2^2=\ell_1^2+\ell_2^2$, the statement of the lemma is equivalent to
\begin{equation*}
\frac{2}{\pi}
\lp \arctan\lp \frac{\ell_2}{\ell_1} \rp + \arctan\lp \frac{m_2}{m_1} \rp + \arctan(\kappa) \rp \overset{}{=} \sgn(\ell_1)\sgn(\ell_2+\kappa \ell_1).
\end{equation*}
This identity may be deduced using general properties of arctangent.  The cases in which one of $\ell_1, \ell_2+\kappa \ell_1$ vanishes can be shown similarly.
\end{proof}

Now, consider a general rank two false theta function
$$ \sum_{\bm n\in\Z^2+{\bm \a}} \sgn(n_1)\sgn(n_2)q^{\frac{1}{2}\left(an_1^2+2bn_1n_2 + cn_2^2\right)},$$
where $a$, $b$, and $c$ are integers such that the quadratic form in the exponent is positive definite, and ${\bm \a } =(\a_1,\a_2) \in \mathbb{Q}^2$. Moreover define the theta functions
\begin{align*}
\Theta_1({\bm w})& :=\sum_{{\bm n} \in \mathbb{Z}^2+{\bm \alpha}} n_1 \left( n_2+\frac{b}{c}n_1\right) e^{\pi i \frac{\DD}{c} n_1^2 w_1+ \pi i c \left(n_2+\frac{b}{c}n_1\right)^2 w_2}, \\ 
\Theta_2({\bm w})& :=\sum_{{\bm n} \in \mathbb{Z}^2+{\bm \alpha} } n_2 \lp n_1+\frac{b}{a}n_2 \rp  e^{\pi i \frac{\DD}{a} n_2^2 w_1  + \pi i a \lp n_1+\frac{b}{a} n_2 \rp^2 w_2}, 
\end{align*}
where $\DD:=ac-b^2 >0$, 
and the modular theta function
$$\Theta(\tau):=\sum_{{\bm n} \in \mathbb{Z}^2+{\bm \alpha} } q^{\frac{1}{2}\left(an_1^2+2b n_1n_2+c n_2^2\right)}.$$
Then we have the following:

\begin{prop}\label{prop:generic_falsetht_sign_lemma} We have
\begin{align*}
& \sum_{{\bm n} \in \mathbb{Z}^2+{\bm \alpha}}{\rm sgn}(n_1){\rm sgn}(n_2)q^{\frac12\left(an_1^2+2b n_1n_2+c n_2^2\right)}  
- \frac{2}{\pi} \d_{\bm{\a} \in \IZ^2} \arctan\left(\frac{b}{\sqrt{\DD}}\right)
\\ & \qquad  \qquad
=\sqrt{\DD} \int_{\tau}^{\tau+i\infty}  \int_{\tau}^{w_1} \frac{\Theta_1({\bm w})+ \Theta_2({\bm w})}{ \sqrt{i (w_1-\tau)} \sqrt{i (w_2-\tau)}}d w_2 d w_1 - \frac{2}{\pi}  \arctan\left(\frac{b}{\sqrt{\DD}}\right) \Theta(\tau),
\end{align*}
where $\d_C=1$ if a condition $C$ holds and zero otherwise.
\end{prop}
\begin{proof}
Letting $\ell_1=\sqrt{\frac{\DD}{c}}n_1$, $\ell_2=\sqrt{c} n_2+\frac{b}{ \sqrt{c}}n_1$,
and $\kappa=-\frac{b}{\sqrt{\DD}}$, we get
$$\textnormal{sgn}(\ell_1)\textnormal{sgn}(\ell_2+\kappa \ell_1)  q^{\frac{\ell_1^2}{2} + \frac{\ell_2^2}{2}}={\rm sgn}(n_1){\rm sgn}(n_2)q^{\frac{1}{2}\left(an_1^2+2b n_1n_2+c n_2^2\right)}.$$
Summing over $\mathbb{Z}^2+{\bm \alpha}$ using Lemma \ref{SignLemma}, noting that $m_1=\sqrt{\frac{\DD}{a}}n_2$ and $m_2=\frac{1}{ \sqrt{a}}(an_1+bn_2)$ and including a correction for the case $(\ell_1,\ell_2+\k \ell_1) =  (0,0)$ which occurs if $\bm \alpha \in \IZ^2$ yields the claim.
\end{proof}

\begin{rem}
We may modify the above construction to get a family of functions for which both the modular part including $ \Theta(\tau)$ and the correction term including $\d_{\bm{\a} \in \IZ^2}$ vanish.
For this purpose, consider false theta functions of the form
$$\sum_{\bm n\in\Z^2+(0,\alpha_2)} (-1)^{n_1} \sgn(n_1)\sgn(n_2)q^{\frac{1}{2}\left(an_1^2+2bn_1n_2 + cn_2^2\right)},$$
such that $a \mid b$ and $\frac{b}{a} \alpha_2 \equiv \frac12 \pmod 1$. In particular, we have $\a_2 \not \in \IZ$ and hence the correction term, with $\d_{\bm{\a} \in \IZ^2}$, vanishes. Note that this condition is satisfied if $\a_2=\frac{1}{2r}$, where $r=\frac{b}{a}$. Some series of this form are discussed in Chapter 5. As in Proposition \ref{prop:generic_falsetht_sign_lemma}, we can represent these $q$-series as iterated 
Eichler-type integrals with $\Theta_1$, $\Theta_2$, and $\Theta$ now picking up an additional $(-1)^{n_1}$ factor. Because $\frac{b}{a}\a_2 \equiv \frac12 \pmod 1$, the corresponding $\Theta$-part is vanishing as
$$\sum_{{\bm n} \in \mathbb{Z}^2+(0,\alpha_2)} (-1)^{n_1} q^{\frac12\left(an_1^2+2bn_1 n_2+c n_2^2\right)}=\sum_{n_2 \in \mathbb{Z}+\alpha_2} q^{\frac{\DD}{2a} n_2^2} \sum_{n_1 \in \mathbb{Z}} (-1)^{n_1}
 q^{\frac{a}{2} \left(n_1+\frac{bn_2}{a}\right)^2} =0.$$
\end{rem}

\section{Decomposition Formulas for Meromorphic Jacobi Forms}\label{sec:meromorphic_Jacobi_decomp}

Before moving to examples, we collect a few auxiliary results used in decomposing multivariable meromorphic Jacobi forms and extracting their Fourier coefficients. We start with a basic result involving two Jacobi theta functions. Besides its use in Section \ref{sec:A_2_parafermion}, the methods employed in its proof are employed as a blueprint for more complex variations that we need in sections below.
Here and throughout we sometimes drop dependencies on $\t$ if they are clear from the context; e.g. we often write $\eta$ instead of $\eta(\tau)$. The next result was suggested to us by S.~Zwegers.

\begin{lem} \label{prop:double_theta_identity}

For $r\in\IZ$ and $w \not\in \IZ\tau+\IZ$ we have
\begin{align*}
\frac{\zeta^r}{\vartheta(z)\vartheta(z+w)} 
=\frac{i}{\eta^3\vartheta(w)}\sum_{n\in\IZ} \frac{q^{n^2-rn}e^{-2\pi i n w }}{1-\zeta q^{n}}
-\frac{ie^{-2\pi i r w}}{\eta^3 \vartheta(w)} \sum_{n\in\IZ} \frac{q^{n^2-rn}e^{2\pi i n w}}{1-\zeta e^{2\pi i w}q^{n}}.
\end{align*}
\end{lem}
\begin{proof}
Define
\begin{equation*}
h(\zz)  :=\frac{e^{2\pi i r \zz}}{\vartheta(\zz)\vartheta(\zz+w)},\qquad
g(z,\zz):= \sum_{n\in\IZ} \frac{q^{n^2-rn}e^{-2\pi i n (2 \zz + w)}}{1-\zeta e^{-2\pi i  \zz}q^{n}}.
\end{equation*}
Using \eqref{JE} gives that $\zz\mapsto h(\zz)g(z,\zz)$ is elliptic.  Let $P_\delta:=\delta+[0,1]+[0,1]\t$ be a fundamental parallelogram with $\delta$ in a small neighborhood of $0$ such that $\zz\mapsto h(\zz)g(z,\zz)$ has no poles on the boundary.  Moreover, we assume that $z$ and $-w$ are in $P_\delta$ and prove the proposition statement for such values; the result generalizes to the whole complex plane by analytic continuation.
If we integrate $h(\zz)g(z,\zz)$ around $P_\delta$ counterclockwise, then the integral vanishes by ellipticity of the function and we have, by the Residue Theorem
\begin{equation*}
0=\int_{\del P_\delta}h(\zz)g(z,\zz)d\zz=2\pi i \sum_{w\in P_\delta}\Res_{\zz=w}(h(\zz)g(z,\zz)).
\end{equation*}
Using that $\Res_{\zz=z}(g(z,\zz))=\frac{1}{2\pi i}$, we get
\begin{align*}
h(z)=-2\pi i  g(z,0) \Res_{\zz=0} (h(\zz))- 2\pi i g(z,-w) \Res_{\zz=-w} (h(\zz)).
\end{align*}
We compute, using \eqref{diffT}
\begin{align*}
\Res_{\zz=0} (h(\zz))
= - \frac{1}{2\pi \eta^3 \vartheta(w)},\qquad
\Res_{\zz=-w} (h(\zz))
= \frac{e^{-2\pi i r w}}{2\pi \eta^3 \vartheta(w)},
\end{align*}
which then gives the claim.
\end{proof}

We next state two variations of this result involving three Jacobi theta functions, which we need in Section \ref{sec:B_2_parafermion} and
whose proofs follow the same method as the one used in Lemma \ref{prop:double_theta_identity}.

\begin{lem}\label{prop:triple_theta_identity}
For $w_1,w_2,w_1-w_2\not\in\Z\tau+\Z,$ and $r \in \Z+\frac{1}{2}$, we have
\begin{align*}
\frac{\zeta^r}{\vartheta(z)\vartheta(z+w_1)\vartheta(z+w_2)} 
&=
\frac{i}{\eta^3 \vt(w_1)\vt(w_2)} 
\sum_{n\in\Z} \frac{(-1)^n q^{\frac{3n^2}{2}-rn} e^{-2\pi i n(w_1+w_2)}}{1- \z q^n}
\\ & \qquad 
+\frac{ie^{-2\pi i rw_1}}{\eta^3 \vt(w_1)\vt(w_1-w_2)} 
\sum_{n\in\Z} \frac{(-1)^n q^{\frac{3n^2}{2}-rn} e^{-2\pi i n(w_2-2w_1)}}{1- \zeta e^{2\pi iw_1}q^n}
\\ & \qquad \qquad
+
\frac{ie^{-2\pi i r w_2}}{\eta^3 \vt(w_2)\vt(w_2-w_1)}  \sum_{n\in\Z} \frac{(-1)^n q^{\frac{3n^2}{2}-rn} e^{-2\pi i n(w_1-2w_2)}}{1- \zeta e^{2\pi i w_2}q^n}.
\end{align*}
\end{lem}

\begin{lem}\label{prop:triple_theta_identity2}
For  $w_1,w_2\not\in\Z\frac{\tau}{2}+\Z\frac{1}{2} $, $w_1-w_2\not\in\Z\tau+\Z$, and $r \in \Z$, we have
\begin{align*}
&\frac{\zeta^r}{\vartheta(2z) \vartheta(z+w_1) \vartheta(z+w_2) }
\\ & = 
\frac{i e^{-2\pi i r w_1}}{\eta^3\vt(2w_1)\vt(w_1-w_2)} 
\sum_{n\in\Z} \frac{ q^{3n^2-rn} e^{2\pi i n(5w_1-w_2)}}{1- \zeta e^{2\pi i w_1}q^n}
+\frac{i e^{-2\pi i r w_2}}{\eta^3\vt(2w_2)\vt(w_2-w_1)} 
\sum_{n\in\Z} \frac{ q^{3n^2-rn} e^{2\pi i n(5w_2-w_1)}}{1- \zeta e^{2\pi i w_2}q^n}
\\&\qquad  
+\frac{i}{2\eta^3}
\sum_{\ell_1,\ell_2\in\{0,1\}}
\frac{(-1)^{\ell_1+\ell_2+r\ell_2}q^{\frac{\ell_1(\ell_1+r)}{2}} }{ 
\vt \lp w_1+\frac{\ell_1\tau+\ell_2}{2}\rp 
\vt\left(w_2+\frac{\ell_1\tau+\ell_2}{2}\right)} 
\sum_{n\in\Z} \frac{ q^{3n^2-(3\ell_1+r)n} e^{-2\pi i n(w_1+w_2)}}{1- (-1)^{\ell_2} \zeta q^{n-\frac{\ell_1}{2}} } .
\end{align*}
\end{lem}

\section{Generic Parafermionic Characters of type $A_2$}\label{sec:A_2_parafermion}

\subsection{Parafermions and parfermion algebras}
The parafermionic conformal field theories first appeared in the famous article of Fateev and Zamolodchikov on $\mathbb{Z}_k$-parafermions \cite{FZ}. The fields in such theories have fractional conformal weight and are not necessarily local to each other, which thereby generalizes the familiar bosonic and fermionic free fields.

In mathematics literature, parafermions and parafermionic spaces originally appeared in the ground-breaking work of Lepowsky and Wilson on $Z$-algebras and Rogers--Ramanujan identities \cite{LW}. This concept was later formalized by Dong and Lepowsky in \cite{DL}, where parafermionic spaces \cite{FZ} were viewed as examples of  generalized vertex algebras. Although \cite{LW,FZ} dealt only with $\mathfrak{sl}_2$ parafermions at positive integral levels, parafermions can be defined for any affine Lie algebra $\mathfrak{g}$ and any level $k$.
In this generality, the parafermionic space $\Omega_k(\mathfrak{g})$ consists of highest weight vectors for the Heisenberg vertex subalgebra inside the affine vertex algebra $V_k(\mathfrak{g})$. 
The {\em parafermion (vertex)  algebra}, denoted by $N_k(\mathfrak{g}) \subset \Omega_k(\mathfrak{g})$, is defined as the charge zero subspace of the parafermionic space. It has a natural vertex operator algebra structure of central charge $c=\frac{k {\rm dim}(\mathfrak{g})}{k+h^\vee}- n$. Then the {\em parafermionic character} is defined by 
\begin{equation*}
\mathrm{ch}[N_k(\mathfrak{g})](q):=\mathrm{tr}|_{N_k(\mathfrak{g})} q^{L(0)-\frac{c}{24}},
\end{equation*}
where $L(0)$ is the degree operator. 
This can in turn be expressed as the constant term (\ref{par_ct}) discussed in the introduction. To illustrate this concept, let us consider the simplest non-trivial case of $V_2(\mathfrak{sl}_2)$. The parafermionic space $\Omega_{2}(\mathfrak{sl}_2)$ is simply the free fermion vertex superalgebra and $N_2(\mathfrak{sl}_2)$ is the even part thereof, also known as the $c=\frac12$ Ising model. Therefore,
\begin{equation*}
\mathrm{ch}[N_2(\mathfrak{sl}_2)](q)=q^{-\frac{1}{48}}
\left(\frac{\left(-q^{\frac12};q\right)_\infty}{2}+\frac{\left(q^{\frac12};q\right)_\infty}{2} \right).
\end{equation*}
For other levels,  $k  \in \mathbb{N}$, $k \geq 3$, the algebraic structure of $N_k(\mathfrak{sl}_2)$ is more complicated and involves non-linear $W$-algebras. Parafermionic characters
of $\mathfrak{sl}_2$ for positive integral levels are well-understood \cite{Andrews,KP} and they transform as vector-valued modular forms of weight zero. Similar results persist for higher rank algebras.

For generic $k$, that is if $V_k(\mathfrak{g})$ is the universal affine vertex algebra (e.g. $k \not\in \mathbb{Q}$), properties of $N_k(\mathfrak{g})$ are quite different. The structure of the parafermion  algebra is known explicitly only in a handful of examples and their parafermionic characters are not modular.

\subsection{Parafermionic character of $A_2$}

We are finally at a point where we can work out our first example involving generic parafermionic characters of type $A_2$. As a warm up to this discussion, we first consider the simplest example, which is the generic parafermionic characters of type $A_1$.

\noindent {\bf Example.} \label{example} For $\mathfrak{g}=\operatorname{sl}_2$, the parafermionic character is known to be (see for instance  \cite{AMW,Andrews})
\begin{equation*}\label{CT}
{\rm CT}_{ [\zeta]} \left(\frac{1}{({\zeta}q ;q)_\infty  ({ \zeta^{-1}}q;q)_\infty}\right)=\frac{1}{(q;q)_\infty^2}\left(-1+2 \sum_{n = 0}^{\infty} (-1)^{n} q^{\frac{n(n+1)}{2}}\right)=-\frac{q^{\frac{1}{12}}}{\eta(\tau)^2} +2\frac{q^{-\frac{1}{24}} \psi(\tau)}{\eta(\tau)^2},
\end{equation*}
where $\psi(\tau):=\sum_{n \in \mathbb{Z}} {\rm sgn}(n+\frac14)q^{2(n+\frac14)^2}$ is Rogers' false theta function. The modular properties of  $\frac{\psi(\tau)}{\eta(\tau)^2}$ were studied and used in \cite{BN} to give a Rademacher type exact formula for its coefficients in the $q$-expansion. The constant term in the above example splits into two $q$-series with different modular behaviors (note the different $q$-powers). Our goal is to
obtain a similar decomposition for the $A_2$ vacuum character. 

\subsection{Expression in terms of false theta functions} \hfill

Specializing equations \eqref{paraf} and \eqref{par_ct} to the case of $A_2$ with positive roots 
\begin{equation*}
\Delta_+:=\left\{\alpha_1 =\left(\begin{matrix} 1 \\0 \end{matrix}\right), \ 
\alpha_2 =\left(\begin{matrix} 0 \\1 \end{matrix}\right), \ 
\alpha_1+\alpha_2\right\},
\end{equation*}
the goal in this section is to study the constant term of 
\begin{equation*}
G(\bm \z) := q^{\frac{8k}{24(k+3)}} (q;q)_\infty^{2} {\rm ch}[V_k(\mathfrak{sl}_3)]({\bm \zeta};q)=
\frac{1}{\left(\zeta_1q,\zeta_1^{-1}q,\zeta_2q,\zeta_2^{-1}q,\zeta_1\zeta_2q,\zeta_1^{-1}\zeta_2^{-1}q;q\right)_\infty},
\end{equation*}
where $(a_1, \dots , a_\ell ;q)_n := \prod_{j=1}^\ell (a_{j};q)_n.$ 
Using \eqref{JT} we rewrite it as ($\zeta_j:=e^{2\pi i z_j}$)
\begin{equation}\label{eq:G_theta_eta}
G(\bm \z) = i q^{\frac{1}{4}} \eta^3 \frac{\z_1^{-1} \z_2^{-1} (1-\z_1)(1-\z_2)(1-\z_1\z_2)}{\vth (z_1) \vth (z_2) \vth (z_1+z_2)} .
\end{equation}

Then, to state our result on the constant term of $G(\bm \z)$, we introduce the following functions:
\begin{align*}
G_0 (\t) & := 1+3 \sum_{n \in \IZ} |n| q^{n^2}  
- 6 q^{-\frac{1}{4}} \sum_{n \in \IZ+\frac{1}{2}} |n| q^{n^2},  \\ 
\Psi (\t) & := \sum_{\bm{n} \in\IZ^2+\lp \frac{1}{3},\frac{1}{3} \rp }  
\sgn (n_1) \sgn (n_2)n_1  q^{Q_A (\bm{n})},
\qquad \mathrm{where} \ \ Q_A(\bm{n}) := n_1^2+n_1n_2+n_2^2 . 
\end{align*}
\begin{prop}\label{prop_const_term_A2}
For $|q| < |\zeta_1|, |\zeta_2|, |\zeta_1 \zeta_2| < 1$ we have
\begin{align*}
\mathrm{CT}_{[\bm\z]} \left(G\left(\bm\z\right)\right) &= 
\frac{q^{\frac{1}{4}} }{\eta(\t)^6 }G_0 (\t)  
+ \frac{9q^{-\frac{1}{12}}}{\eta(\t)^6} \Psi (\t) 
\\ & 
= 1 + 3q^2 + 8 q^3 + 21 q^4 + 48 q^5 + 116 q^6
+252 q^7 + 555 q^8 + 1156 q^9  + O\left(q^{10}\right).
\end{align*}
\end{prop}
To prove Proposition \ref{prop_const_term_A2}, we employ Lemma \ref{prop:double_theta_identity} and another auxiliary result stated below, which itself is a corollary of Lemma \ref{prop:double_theta_identity}.

\begin{lem}\label{cor:theta_squared_identity}
For $r\in\IZ$ we have
\begin{align*}
\frac{\zeta^r}{\vartheta(z)^2} =
-\frac{1}{\eta^6} \sum_{n\in\IZ} q^{n^2-rn}
\lp \frac{2n-r-1}{1-\zeta q^n} + \frac{1}{(1-\zeta q^n)^2} \rp .
\end{align*}
\end{lem}
\begin{proof}
Using \eqref{diffT} and the fact that $\vt$ is odd, we find that for a function $F$ that is holomorphic in a neighborhood of $w=0$, we have
\begin{equation*}
\frac{F(w)}{\vartheta(w)} = -\frac{1}{2\pi \eta^3} \lp \frac{F(0)}{w} + F'(0) \rp + O (w)
\qquad
\mbox{as } w \to 0.
\end{equation*}
Thus taking the limit $w \to 0$ in Lemma \ref{prop:double_theta_identity} yields (noting that $F(0)=0$ in this case)
\begin{align*}
\frac{\zeta^r}{\vartheta(z)^2} = -\frac{i}{2 \pi \eta^6}
\sum_{n\in\IZ} q^{n^2-rn}
\lb \frac{\del}{\del w} \lp \frac{e^{-2\pi i n w }}{1-\zeta^r q^n}
-  \frac{e^{2\pi i (n-r) w}}{1-\zeta^r e^{2\pi i w}q^{n}} \rp \rb_{w=0}.
\end{align*}
The result follows, using that
\begin{equation*}
\frac{i}{2\pi} \lb \frac{\del}{\del w} \lp \frac{e^{-2\pi i n w }}{1-\zeta q^{n}}
-  \frac{e^{2\pi i (n-r) w}}{1-\zeta e^{2\pi i w}q^{n}} \rp \rb_{w=0}
=
\frac{2n-r-1}{1-\zeta q^n} + \frac{1}{(1-\zeta q^n)^2}. \qedhere
\end{equation*}
\end{proof}

We are now ready to compute Fourier coefficients of the meromorphic Jacobi form appearing in equation 
\eqref{eq:G_theta_eta}. To state our result, we define
\begin{align*}
D(\bm{r}) &:= \mathrm{CT}_{[\bm \zeta]} 
\lp \frac{i \eta^9 \zeta_1^{r_1} \zeta_2^{r_2}}{\vth(z_1)\vth(z_2) \vth(z_1+z_2)} \rp,\\
D_1 (\bm{r}) &:= \sum_{\bm{n} \in \IN_0^2}
(n_1+2n_2-r_1) 
q^{n_1^2+n_1n_2+n_2^2-r_2n_1-r_1n_2} ,
\\
D_2 (\bm{r}) &:= 
 \sum_{\bm{n} \in \IN_0^2} 
(n_1-2n_2+r_1-r_2) 
q^{n_1^2-n_1n_2+n_2^2-r_2n_1+(r_2-r_1)n_2}.
\end{align*}

\begin{cor}
\label{cor:constant_term_D}
For $|q| < |\zeta_1|, |\zeta_2|, |\zeta_1 \zeta_2| < 1$ and for $\bm{r} \in \IZ^2$ we have
\begin{align*}
D(\bm r)=D_1(\bm r)+D_2(\bm r).
\end{align*}
\end{cor}
\begin{proof}
Using Lemma \ref{prop:double_theta_identity} with $(r,z,w) \mapsto (r_2,z_2,z_1)$ we find that with $T_A$ defined in \eqref{defineT},
\begin{equation*}
T_A(\bm z) \zeta_2^{r_2}
=\frac{i}{\eta^3\vth(z_1)^2} \lp
\sum_{n_1\in\IZ} \frac{q^{n_1^2-r_2 n_1}\zeta_1^{-n_1}}{1-\zeta_2q^{n_1}}
- \sum_{n_1\in\IZ} \frac{q^{n_1^2-r_2 n_1}\zeta_1^{n_1-r_2}}{1-\zeta_1 \zeta_2q^{n_1}}
\rp.
\end{equation*}
Next, we use Lemma \ref{cor:theta_squared_identity} with $(r,z) \mapsto (r_1-n_1,z_1)$ and $(r,z) \mapsto (r_1-r_2+n_1,z_1)$ to write
\begin{align*}
&i \eta^9 T_A(\bm z) \zeta_1^{r_1} \zeta_2^{r_2}
= \sum_{\bm n \in \IZ^2} \frac{q^{n_1^2+n_1n_2+n_2^2-r_2n_1-r_1n_2}}{1-\zeta_2q^{n_1}}
\lp \frac{2n_2+n_1-r_1-1}{1-\zeta_1q^{n_2}} + \frac{1}{(1-\zeta_1q^{n_2})^2} \rp 
\notag
\\ & \quad
-\sum_{\bm n \in \IZ^2} \frac{q^{n_1^2-n_1n_2+n_2^2-r_2n_1+(r_2-r_1)n_2}}{1-\zeta_1 \zeta_2 q^{n_1}}
\lp \frac{2n_2-n_1+r_2-r_1-1}{1-\zeta_1 q^{n_2}} + \frac{1}{(1-\zeta_1 q^{n_2})^2} \rp .
\label{eq:T_expression_A2}
\end{align*}
The claim now follows using the identity
\begin{equation}\label{eq:constant_term}
\mathrm{CT}_{[\zeta]} \left(\frac{1}{(1-\zeta q^n)^k}\right) = 
\begin{cases}
1 &\mbox{ if } n \geq 0, \\
0 &\mbox{ if } n < 0,
\end{cases}
\end{equation}
which holds for $z$ sufficiently close to $0$ with $|\z| < 1$ and $k \in \mathbb N$.
\end{proof}

We are now ready to prove Proposition \ref{prop_const_term_A2}. 

\begin{proof}[Proof of Proposition \ref{prop_const_term_A2}]
Using \eqref{eq:G_theta_eta} and Corollary \ref{cor:constant_term_D}, for $|q| < |\zeta_1|, |\zeta_2|, |\zeta_1 \zeta_2| < 1$ we have
\begin{equation*}
\mathrm{CT}_{[\bm\z]} \left(G\left(\bm\z\right)\right) = \frac{q^{\frac{1}{4}}}{\eta^6} 
\sum_{\bm{r} \in \mathcal{S}_A} \e_A(\bm{r}) D(\bm{r}),
\end{equation*}
where 
\begin{align*}
\mathcal{S}_A &:= \left\{ (1,0),(0,1),(-1,-1),(-1,0),(0,-1),(1,1) \right \},
\\
\e_A (\bm{r}) &:= \begin{cases}
1 \qquad & \mbox{if } \bm{r} \in \{(1,0),(0,1),(-1,-1)\}, \\
-1 \qquad & \mbox{if } \bm{r} \in \{(-1,0),(0,-1),(1,1)\} .
\end{cases}
\end{align*}
Defining $Q_A^*(\bm{n}) := Q_A(-n_1,n_2)$,  we rewrite $D_1 (\bm{r})$ and $D_2 (\bm{r})$ as 
\begin{align*}
D_1 (\bm{r}) &= q^{-\frac{Q_A^*(\bm{r})}{3}} \sum_{\bm{n} \in \IN_0^2} (n_1+2n_2-r_1)
q^{Q_A \lp n_1+\frac{r_1-2r_2}{3}, n_2+\frac{r_2-2r_1}{3} \rp} ,
\\
D_2 (\bm{r}) &= q^{-\frac{Q_A^*(\bm{r})}{3}} \sum_{\bm{n} \in \IN_0^2} (n_1-2n_2+r_1-r_2)
q^{Q_A^* \lp n_1-\frac{r_1+r_2}{3}, n_2+\frac{r_2-2r_1}{3} \rp}  .
\end{align*}
Then, 
\begin{align*}
q^{\frac{1}{3}} \sum_{\bm{r} \in \mathcal{S}_A} \e_A(\bm{r}) D_1(\bm{r}) 
&= 
\sum_{\bm{n} \in\IN_0^2} 
\Bigg((n_1+2n_2-1) q^{Q_A \lp n_1+\frac{1}{3}, n_2-\frac{2}{3} \rp} 
+ (n_1+2n_2) q^{Q_A \lp n_1-\frac{2}{3}, n_2+\frac{1}{3} \rp} 
\\ & \qquad \qquad \qquad 
+ (n_1+2n_2+1) q^{Q_A \lp n_1+\frac{1}{3}, n_2+\frac{1}{3} \rp}   
-(n_1+2n_2+1) q^{Q_A \lp n_1-\frac{1}{3}, n_2+\frac{2}{3} \rp} 
\\ & \qquad \qquad \qquad 
- (n_1+2n_2) q^{Q_A \lp n_1+\frac{2}{3}, n_2-\frac{1}{3} \rp} 
-(n_1+2n_2-1)  q^{Q_A \lp n_1-\frac{1}{3}, n_2-\frac{1}{3} \rp} 
\Bigg) .
\end{align*}
Shifting either $n_1$ or $n_2$ by one while collecting the one-dimensional boundary terms yields
\begin{align*}
&\sum_{\bm{r} \in \mathcal{S}_A} \e_A(\bm{r}) D_1(\bm{r}) 
=
3 q^{-\frac{1}{3}} 
\sum_{\bm{n} \in\IN_0^2} \lp  (n_1+2n_2+1) q^{Q_A \lp n_1+\frac{1}{3}, n_2+\frac{1}{3} \rp}  
- (n_1+2n_2+2) q^{Q_A \lp n_1+\frac{2}{3}, n_2+\frac{2}{3} \rp}  \rp
\\&  \quad
+ \sum_{n=0}^\infty \lp  (n-1) q^{n^2} + 2n q^{n^2}  - (2n+1) q^{n(n+1)} - n q^{n(n+1)}  - 
(2n-1) q^{n(n-1)} - n q^{n(n+1)} \rp   .
\end{align*}
Changing $\bm n \mapsto -(1,1)-\bm n$ for the second two-dimensional term
and shifting $n \mapsto n+1$ in the one-dimensional contribution with the factor $q^{n(n-1)}$ we find that
\begin{align*}
\sum_{\bm{r} \in \mathcal{S}_A} \e_A(\bm{r}) D_1(\bm{r}) 
&=\frac{3}{2} q^{-\frac{1}{3}} 
\sum_{\bm{n} \in\IZ^2 +\lp \frac{1}{3}, \frac{1}{3} \rp} 
 (1+\sgn (n_1) \sgn (n_2) )(n_1+2n_2)q^{Q_A(\bm n)}
 \\
& \qquad\qquad  \qquad\qquad
+1+ \sum_{n=0}^\infty \lp (3n-1) q^{n^2} - 2(3n+1) q^{n(n+1)}  \rp .
\end{align*}

A similar computation gives
\begin{align*}
\sum_{\bm{r} \in \mathcal{S}_A} \e_A(\bm{r}) D_2(\bm{r})  
=& \frac{3}{2} q^{-\frac{1}{3}}
\sum_{\bm{n} \in\IZ^2+\lp \frac{1}{3},\frac{1}{3} \rp }  
  (-1+\sgn (n_1) \sgn (n_2) )(n_1+2n_2) q^{Q_A (\bm{n})}
  \\ & \qquad\qquad  \qquad\qquad
+ \sum_{n=0}^\infty \lp (3n+1) q^{n^2}  - 2(3n+2) q^{n(n+1)} \rp.
\end{align*}
Then, combining the two terms we find 
\begin{align*}
\sum_{\bm{r} \in \mathcal{S}_A} \e_A(\bm{r}) D(\bm{r}) 
&=
3 q^{-\frac{1}{3}}
\sum_{\bm{n} \in\IZ^2+\lp \frac{1}{3},\frac{1}{3} \rp }  
 \sgn (n_1) \sgn (n_2) (n_1+2n_2)q^{Q_A (\bm{n})}
   \\ & \qquad\qquad  \qquad\qquad
+1+ \sum_{n=0}^\infty \lp 6n q^{n^2}  - 6\lp 2n+1 \rp q^{\lp n+\frac{1}{2} \rp^2 - \frac{1}{4}} \rp.
\end{align*}
Noting the symmetry between $n_1$ and $n_2$ of the two-dimensional sum and antisymmetry of the two one-dimensional sums under $n \mapsto -n$ and $n \mapsto -n-1$, respectively, (as well as the vanishing of the first one-dimensional summand for $n=0$) yields the result.
\end{proof}

\begin{rem} Note that for $\bm r=(r_1,r_2)$, such that $r_1+ r_2 \equiv 0 \pmod 3$, the coefficient $D(\bm r)$ is a finite sum of one-dimensional false theta functions.
Specifically for $k \in \mathbb{N}_0$, we have
\begin{align*}
D(3k,3k) &=
\lp \sum_{j=0}^{k-1} - \sum_{j=k+1}^{3k}  \rp q^{j^2-3kj} \sum_{n =0}^\infty (2n+j-3k) q^{n^2+(j-3k)n}
\\ & \qquad \qquad 
+ \sum_{j=-k}^{k} q^{j^2-3k^2} \sum_{n =0}^\infty q^{n^2} \lp n \frac{q^{jn}+q^{-jn}}{2} + j  \left(q^{jn}-q^{-jn}\right) \rp  .
\end{align*}
In particular, $D(0,0)=\sum_{n = 1}^{\infty} n q^{n^2}$. This leads to the new $q$-hypergeometric representation 
$$
\frac{D(0,0)}{(q)_\infty^6}= \sum_{\bm{n} \in \N_0^4} \frac{q^{n_1+n_2+n_4}}{(q)_{n_1+n_4-n_3-1}(q)_{n_2+n_4-n_3-1}(q)_{n_1} (q)_{n_2}(q)_{n_3}(q)_{n_4}},$$
which easily follows from applying Euler's identity $\frac{1}{(a)_\infty}=\sum_{n = 0}^\infty \frac{a^n}{(q)_n}$ to 
$T_{A}(\bm z;\tau)$ six times.  

Another consequence of the formula for $D(0,0)$ is the following $q$-series identity 
\begin{equation*}
\sum_{n=1}^\infty n q^{n^2}
=\sum_{\bm n\in \mathbb N_0^4} (-1)^{n_1+n_2+n_3}q^{\frac{1}{2}\lp {n_1^2+n_2^2+n_3^2+n_1+n_2+n_3}\rp+(n_1+n_2+n_3+2)(n_4+1)-n_1}\lp 1 + q^{n_1-n_2-n_3-n_4-1}\rp,
\end{equation*}
which follows after three applications of another well-known identity \cite{Andrews}
$$\frac{(q)_\infty^2}{(\zeta)_\infty (\zeta^{-1} q)_\infty}=\sum_{\ell \in \mathbb{Z}} \zeta^{\ell} \sum_{n \geq 0}  (-1)^n q^{\frac{n^2+n}{2}+n| \ell | +\frac{1}{2}(|\ell|-\ell)}.$$
\end{rem}

\subsection{Modular properties of the parafermion character} \hfill

We now study the modular transformations of $\Psi$ appearing in the $A_2$ parafermion  character. This contains a two-dimensional false theta function and is the more interesting part of the character. The first step is to apply Lemma \ref{SignLemma} and rewrite $\Psi$ in a more appropriate form to analyze modular properties. To give this statement, we consider the function
\begin{equation*}
h (\bm{w} )  := \vth_{3,1}^{[1]} (w_1) \vth_{1,1} (w_2) - \vth_{3,2}^{[1]} (w_1) \vth_{1,0} (w_2) 
\end{equation*}
and also define the following regularized integral for $w_1 \in \IH \setminus \{\t\}$:
\begin{equation*}
\rint{\t}{w_1} \frac{f(w_2)}{(i (w_2-\t))^{\frac{3}{2}}} dw_2 := 
\lim_{\zz \to \t}  \lp
\int_\zz^{w_1} \frac{f(w_2)}{(i (w_2-\t))^{\frac{3}{2}}} dw_2+ 2 i \frac{f(\t)}{\sqrt{i(\zz-\t)}}
\rp ,
\end{equation*}
where both the integral and the one-sided limit are taken along the hyperbolic geodesic from $\t$ to $w$. Now, one could   deform the path of integration away from the hyperbolic geodesic and provided that the contour does not cross the branch point at $w_2 = \t$, the value of the regularized integral is maintained thanks to the holomorphy of the integrand. The choice for the path here gives a concrete way to compute the integral while working with the principal value of the square root and moreover is quite convenient for studying the modular transformation properties we encounter in this paper. In fact, for the remainder of this paper we assume that all similar (iterated) integrals in the upper half-plane, including the one-sided limits involved in the regularization, are taken along hyperbolic geodesics.

\begin{prop}\label{prop:Psi_integral_form}
We have
\begin{equation*}
\Psi (\t) = \frac{\sqrt{3}}{2 \pi} \int_\tau^{\tau +i\infty} \rint{\tau}{w_1} 
\frac{h(\bm{w})}
{\sqrt{i(w_1-\tau)} (i(w_2-\tau))^{\frac 32}} dw_2  dw_1 .
\end{equation*}
\end{prop}
\begin{proof}
The claim follows from Lemma \ref{SignLemma} and integration by parts noting that $h(w_1,w_1) = 0$. 
\end{proof}

We now define  the completion of $\Psi$ as a function on $\mathbb{H} \times \mathbb{H}$ by
\begin{equation*}
\wh{\Psi} (\t,w) := \frac{\sqrt{3}}{2 \pi} \int_\tau^{w} \rint{\tau}{w_1} 
\frac{h(\bm{w})}
{\sqrt{i(w_1-\tau)} (i(w_2-\tau))^{\frac 32}} dw_2  dw_1,
\end{equation*}
so that, with the limit taken to be vertical
\begin{equation*}
\Psi (\t) = \lim_{w \to \t +i \infty} \wh{\Psi} (\t,w).
\end{equation*}
Note that, unlike the one-dimensional false theta functions studied in \cite{BN} (where a cut-plane is used for the domain of $w$), the integral to  $i \infty$ can be taken in any direction as long as the same branch of square-root is used for both half-integral powers in the integrand.

\begin{prop}
\label{prop:Psi_modular_properties}
For $M = \left(\begin{smallmatrix}
a & b\\ c & d
\end{smallmatrix}\right)\in \operatorname{SL}_2(\Z)$ we have
\begin{equation*}
\wh{\Psi}  \lp \frac{a \t + b}{c \t + d}, \frac{a w + b}{c w + d} \rp
=
\nu_{\eta}(M)^{8} (c\t +d)^2 \, \wh{\Psi}(\t,w).
\end{equation*}
\end{prop}
\begin{proof}
It suffices to prove the statement for translation and inversion, in which case the claim is 
\begin{equation}\label{reduces}
\wh{\Psi} (\t+1,w+1) = e^{\frac{2\pi i}{3}} \wh{\Psi} (\t,w)
\qquad \mbox{and} \qquad
\wh{\Psi}  \lp -\frac{1}{\t}, -\frac{1}{w} \rp  = \tau^2\wh{\Psi}  \lp \t, w \rp .
\end{equation}

We first recall that the integrals in $w_1$ and $w_2$ (as well as the one-sided limit used in regularizing the integral) are taken along the hyperbolic geodesic from $\tau$ to $w$, i.e., along the unique circle with a real center containing $\tau$ and $w$ or along the straight vertical line from $\tau$ to $w$ if $\im{(\tau)}=\im{(w)}$. Then, we modify $\wh{\Psi} (\t,w)$ in the following way without changing its value
\begin{equation}\label{eq:Psihat_alternative_form}
\wh{\Psi} (\t,w) = \frac{\sqrt{3}}{2 \pi} \int_\tau^{w}  \frac{1}{\sqrt{i(w_1-\tau)}}
\lim_{\zz \to \t} \lp 
\int_{\zz}^{w_1}  \frac{h(\bm{w})}{(i(w_2-\tau))^{\frac 32}} dw_2
+ 2 i  \frac{h(w_1,\zz)}{\sqrt{i(\zz - \t)}} \frac{\zz - w_1}{\t - w_1}
\rp d w_1 .
\end{equation}
In this form, the modular transformation properties may be concluded by the following modular transformations for $h(\bm{w})$, which can be deduced from Lemma \ref{lem:theta_onedim_modular_transf}:
\begin{equation*}
h(\bm w+(1,1)) = e^{\frac{2\pi i}{3}} h(\bm{w}) 
\quad \mbox{and} \quad
h \lp -\frac{1}{w_1}, -\frac{1}{w_2} \rp = w_1^{\frac{3}{2}} w_2^{\frac{1}{2}} h(\bm{w}) .
\end{equation*}
In fact, to work with $\wh{\Psi} ( -\frac{1}{\t}, -\frac{1}{w} )$, we change variables as $w_1 \mapsto -\frac{1}{w_1}, w_2 \mapsto -\frac{1}{w_2}, z \mapsto -\frac{1}{z}, \zz \mapsto -\frac{1}{\zz}$ in equation \eqref{eq:Psihat_alternative_form}.
Note that integration and limit are originally taken along the geodesic from $-\frac{1}{\t}$ to $-\frac{1}{w}$, and the transformations map them to be on the geodesic from $\t$ to $w$.
We then find 
\begin{align*}
&\wh{\Psi}  \lp -\frac{1}{\t}, -\frac{1}{w} \rp = \frac{\sqrt{3}}{2 \pi} 
\int_{\tau}^{w} \frac{\chi_{w_1,\t} \sqrt{w_1} \sqrt{\t} }{\sqrt{i(w_1-\tau)}} 
\\ & \qquad \qquad
\times\lim_{\zz \to \tau} \Bigg(  \int_{\zz}^{w_1}
\chi_{w_2,\t} 
 \frac{w_2^{\frac32} \t^{\frac32} h\lp -\frac{1}{w_1}, -\frac{1}{w_2} \rp}{(i(w_2-\tau))^\frac32}\frac{dw_2}{w_2^2} 
+2i \chi_{\zz,\t}  \frac{\zz^{\frac12} \t^{\frac12} h\lp -\frac{1}{w_1}, -\frac{1}{\zz} \rp}{\sqrt{i(\zz-\tau)}}
\frac{\zz-w_1}{\tau -w_1}
\frac{\t}{\zz}
\Bigg) \frac{dw_1}{w_1^2} ,
\end{align*}
where 
\begin{equation*}
\chi_{\t_1,\t_2} :=
 \sqrt{ \frac{i (\t_1-\t_2)}{\t_1 \t_2} }
 \frac{\sqrt{\t_1} \sqrt{\t_2}}{\sqrt{i (\t_1-\t_2)}} \in \{ -1,+1 \}
\quad \mbox{for } \t_1,\t_2 \in \mathbb{H}
\end{equation*}
keeps track of signs required to work with the principal value of the square-root. Crucially, along the geodesic from $\tau$ to $w$ we have $\chi_{w_1,\t}=\chi_{w_2,\t}=\chi_{\zz,\t}=\chi_{w,\t}$. Using this fact as well as the inversion properties of $h$ we obtain the second identity in \eqref{reduces}.

A similar and easier computation yields the translation property.
\end{proof}

The one-dimensional false theta functions appearing in $G_0$ can be treated as in \cite{BN} by following a similar strategy to Propositions \ref{prop:Psi_integral_form} and \ref{prop:Psi_modular_properties} and using the regularized integral defined there. Its quantum modularity can be studied by a slight adjustment of the argument in \cite[Theorem 1.5]{BN}.

\section{Generic Parafermionic Characters of type $B_2$} \label{sec:B_2_parafermion}
\subsection{Expression in terms of false functions} 
Our next goal is to study parafermionic characters associated to the affine Lie algebra of type $B_2$, which has the positive roots 
\begin{equation*}
\left\{
\alpha_1=\left(\begin{matrix} 1 \\0 \end{matrix}\right), \ 
 \alpha_2=\left(\begin{matrix} 0 \\1 \end{matrix}\right), \ 
\alpha_1+\alpha_2, \ 
2 \alpha_1+\alpha_2\right\}.
\end{equation*}
 In particular, specializing equations \eqref{paraf} and \eqref{par_ct} to the case of $B_2$, so that $\mathfrak{g}=\mathfrak{so}_5$, we study the constant term of 
\begin{equation*}
F(\bm\z) := q^{\frac{10 k}{24(k+3)}}(q;q)_\infty^{2} {\rm ch}[V_k(\mathfrak{so}_5)]({\bm \zeta};q)=
\frac{1}{\left(\zeta_1q,\zeta_1^{-1}q,\zeta_2q,\zeta_2^{-1}q,\zeta_1\zeta_2q,\zeta_1^{-1}\zeta_2^{-1}q,\zeta_1^2 \zeta_2q,\zeta_1^{-2}\zeta_2^{-1}q ;q\right)_\infty}.
\end{equation*}
Using equation \eqref{JT} we obtain 
\begin{equation*}
F(\bm\z) = q^{\frac{1}{3}} \eta^4
\frac{\zeta_1^{-2} \zeta_2^{-\frac{3}{2}}(1-\zeta_1)(1-\zeta_2)(1-\zeta_1\zeta_2)\left(1-\zeta_1^2\zeta_2\right)}{\vth(z_1)\vth(z_2)\vth(z_1+z_2)\vth(2z_1+z_2)} .
\end{equation*}
Now, to state our main result for the associated character, we first require some notation. Using the quadratic form  $Q_B(\bm n):=\frac 32n_1^2+3n_1n_2+3n_2^2$, we define
\begin{equation*}
\Phi (\t) := \Phi_1 (\t) + \Phi_2 (\t),
\end{equation*}
where
\begin{align*}
\Phi_1 (\t) &:=\sum_{\bm n\in\Z^2+\left(\frac 13,\frac 16\right)} (-1)^{n_1-\frac 13}
\lp \sgn(n_2) + \sgn (n_1 + n_2) \rp \sgn(n_1) 
 \lp \lp n_1+2n_2 \rp^2 - \frac{E_2 (\t)}{18} \rp
q^{Q_B(\bm n)},
\\
\Phi_2 (\t) &:=\sum_{\bm n\in\Z^2+\left(\frac 13,\frac 16\right)} (-1)^{n_1-\frac 13}
\sgn (n_1+n_2) \sgn (n_2)
 n_1(n_1+2n_2)
q^{Q_B(\bm n)} .
\end{align*}
Then for $\bm a \in\mathbb Z^2$, we let 
\begin{equation*}
\LL_{\bm a} (\t) :=\sum_{\bm n\in \mathbb Z^2+\left(\frac{1}{3},\frac{a_1}{2}+\frac{1}{6}\right)}(-1)^{(a_2+1)\left(n_1-\frac{1}{3}\right)} \sgn(n_1)\sgn(n_1+2n_2)q^{Q_B(\bm n)} .
\end{equation*}
We also need the following one-dimensional false theta functions:
\begin{equation*}
\phi_{r} (\t) :=  \sum_{n \in \Z+\frac{r}{6}} \sgn (n) q^{3 n^2}
\quad \mbox{and} \quad
\w_{r} (\t) :=  
\sum_{n \in \Z+\frac{r}{3} + \frac{1}{2}} (-1)^{n-\frac{r}{3}-\frac{1}{2}} \sgn (n)  q^{\frac{3n^2}{2}}
\end{equation*}
to define 
\begin{align*}
F_0 (\t) :=&  \frac{E_2 (\t)+2}{4} + \frac{\eta(\tau)^6}{\vth \lp \frac{1}{2}; \tau \rp^2}
+ 6 q^{-\frac{1}{24}} \DDD_{\frac{1}{2}}( \w_1 (\t))
- 6 q^{-\frac{3}{8}}  \DDD_{\frac{1}{2}} (\w_0 (\t))
\\ &
+q^{-\frac{1}{12}} \lp 6 \DDD_{\frac{1}{2}} 
- \frac{\eta(\tau)^6}{\vth \lp \frac{1}{2}; \tau \rp^2} 
+  q^{-\frac{1}{2}} \frac{\eta(\tau)^6}{\vth \lp \frac{\t}{2}; \tau \rp^2} 
- q^{-\frac{1}{2}}  \frac{\eta(\tau)^6}{\vth \lp \frac{\tau+1}{2}; \tau \rp^2}
\rp (\phi_1 (\t))
\\ &
- q^{-\frac{1}{3}} \lp 6  \DDD_{\frac{1}{2}} 
+ \frac{\eta(\tau)^6}{\vth \lp \frac{1}{2}; \tau \rp^2} 
+  \frac{\eta(\tau)^6}{\vth \lp \frac{\t}{2}; \tau \rp^2} 
+ \frac{\eta(\tau)^6}{\vth \lp \frac{\tau+1}{2}; \tau \rp^2}
\rp( \phi_2 (\t)).
\end{align*} 
With these definitions at hand, we can give our result.
\begin{prop}\label{prop_const_term_B2}
In the range  $|q| < |\zeta_1^2|, |\zeta_2|,  |\zeta_1 \zeta_2 |,  |\zeta_1^2 \zeta_2 | <1$, we have
\begin{align*}
\operatorname{CT}_{[\bm\z]}(F(\bm\z)) &=
\frac{q^{\frac{1}{3}}}{\eta(\tau)^{8}}F_0(\t)
+\frac{9 q^{-\frac{1}{12}} }{2 \eta(\tau)^{8}} \Phi (\t)
+\frac{q^{-\frac{1}{12}}}{\eta(\tau)^2} \lp
\frac{ \LL_{0,1}(\t)}{\vth \lp \frac{1}{2}; \tau \rp^2}
+ \frac{q^{-\frac{1}{4}} \LL_{1,0}(\t)}{\vth \lp \frac{\tau}{2}; \tau \rp^2}
- \frac{q^{-\frac{1}{4}} \LL_{1,1}(\t)}{\vth \lp \frac{\tau+1}{2}; \tau \rp^2}
\rp
\\ & 
= 1 + 4q^2 + 12 q^3 + 38 q^4 + 100 q^5 + 276 q^6
+688 q^7 + 1709 q^8 + 4020 q^9  + O\left(q^{10}\right).
\end{align*}
\end{prop}
To prove Proposition \ref{prop_const_term_B2}, we require several preliminary results based on Lemmas \ref{prop:triple_theta_identity} and \ref{prop:triple_theta_identity2}. The first of these is an auxiliary statement that helps us study various limits of the two lemmas and it follows from equation \eqref{diffT} and the identity $\vartheta^{(3)}(0)=2\pi^3 \eta^3E_2$.
\begin{lem}\label{lem:const_term_thetadiv}
For a function $F$ that is holomorphic in a neighborhood of $w=0$ and for $a,b \in \R$ we have, as $w \to 0$
\begin{equation*}\label{eq:const_term_thetadiv}
\frac{F(w)}{\vartheta(aw)\vartheta(bw)}=\frac{1}{4\pi^2ab \eta^6}
\lp \frac{F(0)}{w^2} + \frac{F'(0)}{w} \rp 
+
\lp   \frac{a^2+b^2}{ab} \frac{E_2}{24\eta^6}F(0) 
+\frac{F''(0)}{8\pi^2 ab \eta^6} \rp
+O(w).
\end{equation*}
\end{lem}
The next two results are then two particular limits of Lemmas \ref{prop:triple_theta_identity} and \ref{prop:triple_theta_identity2}, respectively, that appear in the proof of Proposition \ref{prop_const_term_B2}. Taking $\bm w=(w,-w)$ in Lemma \ref{prop:triple_theta_identity} and then letting $w\to 0$ using Lemma \ref{lem:const_term_thetadiv}, yields the following statement.
\begin{lem}
\label{cor:thetacubed_inverse}
For $r \in \Z+\frac{1}{2}$ we have
\begin{equation*}
\frac{\zeta^r}{\vth (z)^3}   = 
- \frac{i}{\eta^9} 
\sum_{n \in \Z} (-1)^n q^{\frac{3n^2}{2}- rn} \left(
\frac{4(3n-r-1)^2-E_2}{8(1 - \zeta q^n)}
+ \frac{6n-2r-3}{2(1 - \zeta q^n)^2} + \frac{1}{(1 - \zeta q^n)^3}
\right) .
\end{equation*}
\end{lem}
Plugging in $\bm w=(w,-w)$ in Lemma \ref{prop:triple_theta_identity2} and taking $w \to 0$ using Lemma \ref{lem:const_term_thetadiv}, we obtain the following result.
\begin{lem}
\label{cor:thetacubed_inverse2}
For $r \in \Z$ we have
\begin{align*}
\frac{\zeta^r}{\vth (z)^2 \vth (2z)}   = 
-\frac{i}{\eta^9} 
&\sum_{n \in \Z} q^{3n^2-rn} 
\left( \vphantom{ - \frac{\eta^6}{2}\left.
\sum_{\substack{\ell_1,\ell_2\in\{0,1\} \\ \bm{\ell} \neq \bm{0} }}
\frac{1}{\vt \lp \frac{\ell_1\tau+\ell_2}{2}\rp^2} 
\frac{(-1)^{\ell_1+\ell_2+r\ell_2}q^{\frac{\ell_1(\ell_1-r)}{2}+3\ell_1n} }{1- (-1)^{\ell_2} \zeta q^{n+\frac{\ell_1}{2}}} 
\right)}
\frac{2\lp 6n -r-1 \rp^2-{E_2}}{8\left(1 - \zeta q^n\right)}
+\frac{12n-2r-3}{4(1 - \zeta q^n)^2}
+ \frac{1}{2(1 - \zeta q^n)^3} \right.
\\ & \qquad \left.
- \frac{\eta^6}{2}
\sum_{\substack{\ell_1,\ell_2\in\{0,1\} \\ \bm{\ell} \neq (0,0) }}
\frac{1}{\vt \lp \frac{\ell_1\tau+\ell_2}{2}\rp^2} 
\frac{(-1)^{\ell_1+\ell_2+r\ell_2}q^{\frac{\ell_1(\ell_1-r)}{2}+3\ell_1n} }{1- (-1)^{\ell_2} \zeta q^{n+\frac{\ell_1}{2}}} 
\right).
\end{align*}
\end{lem}

We are now ready to prove Proposition \ref{prop_const_term_B2}.
\begin{proof}[{Proof of Proposition \ref{prop_const_term_B2}}]
Let for $r_1\in\mathbb Z, r_2\in\mathbb Z+\frac{1}{2}$ and $T_B$ defined in \eqref{defineT}
\begin{equation*}
C(\bm r):=\mathrm{CT}_{[\bm\z]} \left( T_B(\bm z) \eta^{12} \zeta_1^{r_1}\zeta_2^{r_2}\right) .
\end{equation*}
Using Lemma \ref{prop:triple_theta_identity} with $(r,z,w_1,w_2) \mapsto (r_2,z_2,z_1,2z_1)$, Lemma \ref{cor:thetacubed_inverse2} with $(z,r) \mapsto (z_1,r_1-3k)$ and $(z,r) \mapsto (z_1,r_1-2r_2+3k)$, and Lemma \ref{cor:thetacubed_inverse} with $(z,r) \mapsto (z_1,r_1-r_2)$ we obtain 
{
\begin{align*}
&T_B(\bm z) \eta^{12} \z_1^{r_1}\z_2^{r_2}
\\ &\ 
= 
\sum_{\bm n \in \Z^2} 
\frac{(-1)^{n_1} q^{\frac{3n_1^2}{2} +3n_1n_2+3n_2^2-r_2n_1-r_1n_2}}{1-\z_2 q^{n_1}}
\left(  \vphantom{\sum_{\substack{\ell_1,\ell_2\in\{0,1\} \\ \bm{\ell} \neq (0,0) }}
\frac{\eta^6}{2\vt \lp \frac{\ell_1\tau+\ell_2}{2}\rp^2} 
\frac{(-1)^{\ell_1+\ell_2+(1+r_1+n_1)\ell_2}q^{\frac{\ell_1\left(\ell_1+3n_1-r_1 \right)}{2}+3\ell_1n_2} }{1- (-1)^{\ell_2} \z_1 q^{n_2+\frac{\ell_1}{2}}}}
\frac{2\lp 3n_1 -r_1+6n_2-1 \rp^2-{E_2}}{8\lp 1 - \z_1 q^{n_2}\rp}
+  \frac{6n_1+{12n_2-2r_1-3}}{4\left(1 - \z_1 q^{n_2} \right)^2} \right.
\\ & \qquad \qquad \left.
+ \frac{1}{2\left(1 - \z_1 q^{n_2}\right)^3} 
- 
\sum_{\substack{\ell_1,\ell_2\in\{0,1\} \\ \bm{\ell} \neq (0,0) }}
\frac{\eta^6}{2\vt \lp \frac{\ell_1\tau+\ell_2}{2}\rp^2} 
\frac{(-1)^{\ell_1+\ell_2+(1+r_1+n_1)\ell_2}q^{\frac{\ell_1\left(\ell_1+3n_1-r_1 \right)}{2}+3\ell_1n_2} }{1- (-1)^{\ell_2} \z_1 q^{n_2+\frac{\ell_1}{2}}}
\right)
\\ &\ 
+
\sum_{\bm n\in \Z^2} 
\frac{(-1)^{n_1} q^{\frac{3n_1^2}{2}-3n_1n_2+3n_2^2-r_2n_1+(2r_2-r_1)n_2}}{1-\zeta_1^2\zeta_2q^{n_1}}
\left( \vphantom{\sum_{\substack{\ell_1,\ell_2\in\{0,1\} \\ \bm{\ell} \neq \bm{0} }}
\frac{\eta^6}{2\vt \lp \frac{\ell_1\tau+\ell_2}{2}\rp^2} 
\frac{(-1)^{\ell_1+\ell_2+(1+r_1+2r_2+n_1)\ell_2}q^{\frac{\ell_1(\ell_1-3n_1-r_1+2r_2)}{2}+3\ell_1n_2} }{1- (-1)^{\ell_2} \z_1 q^{n_2+\frac{\ell_1}{2}}}}
\frac{2\lp 3n_1-6n_2+r_1-2r_2+1 \rp^2-{E_2}}{8\lp1 - \z_1 q^{n_2}\rp} \right.
\\ & \qquad \qquad  \qquad  \qquad \left.
-  \frac{6n_1-12n_2+2r_1-4r_2+3}{4\left(1 - \z_1 q^{n_2}\right)^2}
+ \frac{1}{2(1 - \z_1 q^{n_2})^3} \right.
\\ & \qquad \qquad  \qquad \left.
- 
\sum_{\substack{\ell_1,\ell_2\in\{0,1\} \\ \bm{\ell} \neq (0,0) }}
\frac{\eta^6}{2\vt \lp \frac{\ell_1\tau+\ell_2}{2}\rp^2} 
\frac{(-1)^{\ell_1+\ell_2+(1+r_1+2r_2+n_1)\ell_2}q^{\frac{\ell_1(\ell_1-3n_1-r_1+2r_2)}{2}+3\ell_1n_2} }{1- (-1)^{\ell_2} \z_1 q^{n_2+\frac{\ell_1}{2}}}
\right)
\\ &\ 
-\sum_{\bm n \in \Z^2}  \frac{(-1)^{n_1+n_2} q^{\frac{3n_1^2}{2}+\frac{3n_2^2}{2}-r_2n_1+(r_2-r_1)n_2}}{1-\z_1\z_2q^{n_1}}
\\ &\qquad \qquad \qquad \times
\left(
\frac{4(3n_2-r_1+r_2-1)^2-E_2}{8(1 - \zeta_1 q^{n_2})}
+ \frac{6n_2-2r_1+2r_2-3}{2(1 - \z_1 q^{n_2})^2} + \frac{1}{(1 - \zeta_1 q^{n_2})^3}
\right) .
\end{align*}
}

In the range  $|q|^{\frac{1}{2}} < |\zeta_1| <1$, $|q| < |\zeta_2|,  |\zeta_1 \zeta_2 |,  |\zeta_1^2 \zeta_2 | <1$, we use \eqref{eq:constant_term} to find the constant term as
\begin{equation*}
C(\bm{r}) = \sum_{\ell=1}^6 C_\ell (\bm{r}) + \sum_{\substack{\ell_1,\ell_2\in\{0,1\} \\ \bm{\ell} \neq (0,0) }} 
\lp C_{7,\bm\ell}(\bm r) + C_{8,\bm\ell}(\bm r) \rp,
\end{equation*}
where 
{
\allowdisplaybreaks
\begin{align*}
C_1(\bm{r}) &:= \frac{1}{4} \sum_{\bm n\in\N_0^2} (-1)^{n_1}
(6n_2+3{n_1}-r_1)^2
q^{\frac{3n_1^2}{2}+3n_1n_2+3n_2^2-r_2n_1-r_1n_2}  ,
\\
C_2(\bm{r}) &:= \frac{1}{4} \sum_{\bm n\in\N_0^2} 
(-1)^{n_1} (6n_2-3{n_1}+2r_2-r_1)^2
q^{\frac{3n_1^2}{2}-3n_1n_2+3n_2^2-r_2n_1+(2r_2-r_1)n_2}  
 ,
\\
C_3(\bm{r}) &:= - \frac{1}{2} \sum_{\bm n\in\N_0^2}  
(-1)^{{n_1}+n_2} (3n_2+r_2-r_1)^2 q^{\frac{3n_1^2}{2}+\frac{3n_2^2}{2}- r_2 {n_1} +(r_2-r_1)n_2}
,
\\
C_4(\bm{r}) &:= -\frac{E_2}{8} \sum_{\bm n\in\N_0^2} 
(-1)^{n_1} q^{\frac{3n_1^2}{2}+3n_1n_2+3n_2^2 -r_2 {n_1}-r_1n_2} ,
\\
C_5(\bm{r}) &:= -\frac{E_2}{8} \sum_{\bm n\in\N_0^2} 
(-1)^{n_1} q^{\frac{3n_1^2}{2}-3n_1n_2+3n_2^2 -r_2 {n_1}+(2r_2-r_1)n_2},
\\
C_6(\bm{r}) &:= \frac{E_2}{8} \sum_{\bm n\in\N_0^2}  
(-1)^{{n_1}+n_2} q^{\frac{3n_1^2}{2}+\frac{3n_2^2}{2}- r_2 {n_1} +(r_2-r_1)n_2} ,
\\
C_{7,\bm\ell}(\bm r) &:=-
\frac{\eta^6(-1)^{\ell_1+(r_1+1)\ell_2}}{2\vt\left(\frac{\ell_1\tau+\ell_2}{2}\right)^2}q^{\frac{\ell_1(\ell_1-r_1)}{2}}
\sum_{\bm n\in\N_0^2}(-1)^{(\ell_2+1)n_1}q^{\frac{3n_1^2}{2}+3n_1n_2+3n_2^2+\left(\frac{3\ell_1}{2}-r_2\right)n_1+(3\ell_1-r_1)n_2} ,
\\
C_{8,\bm\ell}(\bm r)&:=-
\frac{\eta^6(-1)^{\ell_1+r_1 \ell_2}}{2\vt\left(\frac{\ell_1\tau+\ell_2}{2}\right)^2}q^{\frac{\ell_1(\ell_1-r_1)}{2}+\ell_1r_2}
\sum_{\bm n\in\N_0^2}(-1)^{(1+\ell_2)n_1 }q^{\frac{3n_1^2}{2}-3n_1n_2+3n_2^2-\left(\frac{3\ell_1}{2}+r_2\right)n_1+(3\ell_1+2r_2-r_1)n_2}.
\end{align*}
}

Then we may write
\begin{align*}
\operatorname{CT}_{[\bm\z]} (F(\bm\z)) = \frac{q^{\frac{1}{3}}}{\eta^{8}}\sum_{\bm r \in\mathcal S_B} \varepsilon_B(\bm r)C(\bm r),
\end{align*}
where
\begin{align*}
\mathcal{S}_B &:= \left\{
\lp -2, -\frac{3}{2} \rp, \lp -1, \frac{1}{2} \rp,\lp 1, -\frac{1}{2} \rp,    \lp 2, \frac{3}{2} \rp, 
\lp -1, -\frac{3}{2} \rp, \lp -2, -\frac{1}{2} \rp, \lp 1, \frac{3}{2} \rp,  
\lp 2, \frac{1}{2} \rp
\right\},\\
\e_B (\bm r) &:= 
\begin{cases}
1 &\mbox{if } \bm{r} \in 
\left\{ \lp -2, -\frac{3}{2} \rp,  \lp -1, \frac{1}{2} \rp, \lp 1, -\frac{1}{2} \rp, \lp 2, \frac{3}{2} \rp \right\}, 
\\
-1 &\mbox{if } \bm{r} \in 
\left\{ \lp -1, -\frac{3}{2} \rp, 
\lp -2, -\frac{1}{2} \rp, \lp 1, \frac{3}{2} \rp,   \lp 2, \frac{1}{2} \rp \right\} .
\end{cases}
\end{align*}

Next, we simplify the individual terms in the decomposition of $C(\bm{r})$. We start with
\begin{align*}
C_6 (\bm r) =\frac{E_2}{8} q^{-\frac{1}{6} \left(r_1^2-2r_1r_2+2r_2^2 \right)}
\sum_{\bm n\in\N_0^2}  
(-1)^{n_1+n_2} q^{\frac{3}{2} \lp n_1-\frac{r_2}{3} \rp^2 + \frac{3}{2} {\lp n_2 -\frac{r_1-r_2}{3} \rp^2}}.
\end{align*}
Note that the sum in $C_6 (\bm r)$ is invariant if $n_1$ and $n_2$ are interchanged together with their respective shifts.  The terms $C_6 (\bm r)$ cancel in pairs  and we have
\begin{equation*}
\sum_{\bm r \in \cS_B} \e_B (\bm r)  C_6 (\bm r) = 0.
\end{equation*}

For the remaining pieces, the details are quite lengthy. Therefore, we carry them out only for one of the terms and leave the remaining ones to the reader. 
We focus on 
\begin{align*}
C_3 (\bm r) = 
-\frac{9}{2} q^{-\frac{1}{6} \left(r_1^2-2r_1r_2+2r_2^2 \right)} \sum_{\bm n\in\N_0^2}  
(-1)^{n_1+n_2}
\lp n_2 +\frac{r_2-r_1}{3} \rp^2
 q^{\frac{3}{2} \lp n_1-\frac{r_2}{3} \rp^2 + \frac{3}{2} \lp n_2 +\frac{r_2-r_1}{3} \rp^2}.
\end{align*}
We have 
\begin{align*}
&C_3 \lp 1, - \frac{1}{2} \rp - C_3 \lp -2, - \frac{1}{2} \rp 
\\ & \qquad
= -\frac{9}{2} q^{-\frac{5}{12}}
\sum_{\bm n\in\N_0^2}  (-1)^{n_1+n_2} q^{\frac{3}{2} \lp n_1+\frac{1}{6} \rp^2}
\lp 
\lp n_2 -\frac{1}{2} \rp^2
q^{\frac{3}{2} \lp n_2 -\frac{1}{2} \rp^2}  
-
\lp n_2 +\frac{1}{2} \rp^2
q^{\frac{3}{2} \lp n_2 +\frac{1}{2} \rp^2}  
\rp
\\ & \qquad
= -\frac{9}{8} q^{-\frac{1}{24}}
\sum_{n_1 = 0}^\infty  (-1)^{n_1} q^{\frac{3}{2} \lp n_1+\frac{1}{6} \rp^2}
+9 q^{-\frac{5}{12}}
\sum_{\bm n\in\N_0^2}  (-1)^{n_1+n_2} 
\lp n_2 +\frac{1}{2} \rp^2
q^{\frac{3}{2} \lp n_1+\frac{1}{6} \rp^2+ \frac{3}{2} \lp n_2 +\frac{1}{2} \rp^2},
\end{align*}
where for the last equality we split off the $n_2=0$ contribution from the first term and then shift $n_2\mapsto n_2+1$ there. 

Next, we have
\begin{align*}
&C_3 \lp -1,  \frac{1}{2} \rp - C_3 \lp 2,  \frac{1}{2} \rp 
\\ & \qquad
= -\frac{9}{2} q^{-\frac{5}{12}}
\sum_{\bm n\in\N_0^2}  (-1)^{n_1+n_2} q^{\frac{3}{2} \lp n_1-\frac{1}{6} \rp^2}
\lp 
\lp n_2 +\frac{1}{2} \rp^2
q^{\frac{3}{2} \lp n_2 +\frac{1}{2} \rp^2}  
-
\lp n_2 -\frac{1}{2} \rp^2
q^{\frac{3}{2} \lp n_2 -\frac{1}{2} \rp^2}  
\rp
\\ & \qquad
= \frac{9}{8} q^{-\frac{1}{24}}
\sum_{n_1 = 0}^\infty  (-1)^{n_1} q^{\frac{3}{2} \lp n_1-\frac{1}{6} \rp^2}
-9 q^{-\frac{5}{12}}
\sum_{\bm n\in\N_0^2}  (-1)^{n_1+n_2} 
\lp n_2 +\frac{1}{2} \rp^2
q^{\frac{3}{2} \lp n_1-\frac{1}{6} \rp^2+ \frac{3}{2} \lp n_2 +\frac{1}{2} \rp^2},
\end{align*}
where for the last equality we split off the $n_2=0$ part from the second term and then shift $n_2 \mapsto n_2+1$ there. Splitting off the terms with $n_1=0$ from the second sum and then changing $n_1 \mapsto -n_1$ and $n_2 \mapsto -n_2-1$ we get
\begin{align*}
C_3 \lp -1,  \frac{1}{2} \rp - C_3 \lp 2,  \frac{1}{2} \rp 
&=
\frac{9}{8} q^{-\frac{1}{24}}
\sum_{n_1 = 0}^\infty  (-1)^{n_1} q^{\frac{3}{2} \lp n_1-\frac{1}{6} \rp^2}
-9 q^{-\frac{3}{8}}
\sum_{n_2 = 0}^\infty  (-1)^{n_2} 
\lp n_2 +\frac{1}{2} \rp^2
q^{\frac{3}{2} \lp n_2 +\frac{1}{2} \rp^2} 
\\ & \qquad
+ 9 q^{-\frac{5}{12}}
\sum_{\bm n\in-\N^2}  (-1)^{n_1+n_2} 
 \lp n_2 +\frac{1}{2} \rp^2
q^{\frac{3}{2} \lp n_1+\frac{1}{6} \rp^2+ \frac{3}{2} \lp n_2 +\frac{1}{2} \rp^2} .
\end{align*}

\noindent Next, we study
\begin{align*}
&C_3 \lp -2, - \frac{3}{2} \rp - C_3 \lp 1, \frac{3}{2} \rp 
\\ & 
= -\frac{9}{2} q^{-\frac{5}{12}}
\sum_{\bm n\in\N_0^2}  (-1)^{n_1+n_2}  \lp n_2 +\frac{1}{6} \rp^2 
q^{\frac{3}{2} \lp n_2 +\frac{1}{6} \rp^2}
\lp 
q^{\frac{3}{2} \lp n_1 +\frac{1}{2} \rp^2} 
-
q^{\frac{3}{2} \lp  n_1 -\frac{1}{2}  \rp^2} 
\rp
\\ & 
=  \frac{9}{2} q^{-\frac{1}{24}}
\sum_{n_2 = 0}^\infty  (-1)^{n_2}  \lp n_2 +\frac{1}{6} \rp^2 
q^{\frac{3}{2} \lp n_2 +\frac{1}{6} \rp^2}
-9 q^{-\frac{5}{12}}
\sum_{\bm n\in\N_0^2}  (-1)^{n_1+n_2}  \lp n_2 +\frac{1}{6} \rp^2 
q^{\frac{3}{2} \lp n_1 +\frac{1}{2} \rp^2+\frac{3}{2} \lp n_2 +\frac{1}{6} \rp^2},
\end{align*}
where for the last equality we split off the  $n_1=0$ contribution from the second term and then shift $n_1 \mapsto n_1+1$ there. Finally, we study
\begin{align*}
&C_3 \lp 2,  \frac{3}{2} \rp - C_3 \lp -1, -\frac{3}{2} \rp 
\\ &  \qquad  \qquad 
= -\frac{9}{2} q^{-\frac{5}{12}}
\sum_{\bm n\in\N_0^2}  (-1)^{n_1+n_2}  \lp n_2 -\frac{1}{6} \rp^2 
q^{\frac{3}{2} \lp n_2 -\frac{1}{6} \rp^2}
\lp 
q^{\frac{3}{2} \lp n_1 -\frac{1}{2} \rp^2} 
-
q^{\frac{3}{2} \lp  n_1 +\frac{1}{2}  \rp^2} 
\rp
\\ &  \qquad  \qquad 
=
-\frac{9}{2} q^{-\frac{1}{24}}
\sum_{n_2 = 0}^\infty  (-1)^{n_2}  \lp n_2 -\frac{1}{6} \rp^2 
q^{\frac{3}{2} \lp n_2 - \frac{1}{6} \rp^2}
\\ & \qquad  \qquad  \qquad  \qquad  \qquad  \qquad 
+  
9 q^{-\frac{5}{12}}
\sum_{\bm n\in\N_0^2}  (-1)^{n_1+n_2}  \lp n_2 -\frac{1}{6} \rp^2 
q^{\frac{3}{2} \lp n_1 +\frac{1}{2} \rp^2+\frac{3}{2} \lp n_2 -\frac{1}{6} \rp^2},
\end{align*}
where for the last equality we split off the $n_1=0$ part from the first term and then shift $n_1 \mapsto n_1+1$ there. Splitting off the $n_2=0$ contribution  from the double sum and then changing $n_2 \mapsto -n_2$ and $n_1 \mapsto -n_1-1$ we get
\begin{multline*}
C_3 \lp 2,  \frac{3}{2} \rp - C_3 \lp -1, -\frac{3}{2} \rp  
=
-\frac{9}{2} q^{-\frac{1}{24}}
\sum_{n_2 = 0}^\infty  (-1)^{n_2}  \lp n_2 -\frac{1}{6} \rp^2 
q^{\frac{3}{2} \lp n_2 - \frac{1}{6} \rp^2}
\\ 
+\frac{1}{4} q^{-\frac{3}{8}}
\sum_{n_1 = 0}^\infty  (-1)^{n_1}  q^{\frac{3}{2} \lp n_1 +\frac{1}{2} \rp^2}
-9 q^{-\frac{5}{12}}
\sum_{\bm n\in-\N^2}  (-1)^{n_1+n_2}  \lp n_2 + \frac{1}{6} \rp^2 
q^{\frac{3}{2} \lp n_1 +\frac{1}{2} \rp^2+\frac{3}{2} \lp n_2 +\frac{1}{6} \rp^2}.
\end{multline*}
\noindent Therefore, the two-dimensional contribution in $\sum_{\bm r \in \cS_B} \e_B (\bm r) C_3 (\bm r)$ is 
\begin{equation*}
\frac{9  q^{-\frac{5}{12}}}{2}
\sum_{\bm n\in \Z^2 + \lp \frac{1}{2}, \frac{1}{6} \rp}  
(-1)^{n_1-\frac 12+n_2-\frac 16}	
\lp 1 + \sgn (n_1) \sgn (n_2) \rp\lp n_1^2 - n_2^2 \rp
q^{\frac{3n_1^2}{2}+\frac{3n_2^2}{2}} .
\end{equation*}
Under $n_1 \mapsto -n_1$ the contribution of the ``$1$'' in parentheses picks  a minus sign (note that $(-1)^{2n_1} = -1$) and hence vanishes. Then we can rewrite the two-dimensional part as
\begin{equation*}
\frac{9  q^{-\frac{5}{12}}}{2} \sum_{\bm n\in \Z^2 + \lp \frac{1}{2}, \frac{1}{6} \rp}  
(-1)^{n_1-\frac 12+n_2-\frac 16} 
\sgn (n_1) \sgn (n_2) \lp n_1^2 - n_2^2 \rp
q^{\frac{3n_1^2}{2}+\frac{3n_2^2}{2}}. 
\end{equation*}
Mapping $\bm n \mapsto (-n_1-n_2, n_2)$ we get 
\begin{equation*}
\frac{9  q^{-\frac{5}{12}}}{2} \sum_{\bm n\in \Z^2 + \lp \frac{1}{3}, \frac{1}{6} \rp}  
(-1)^{n_1-\frac 13} \sgn (n_1+n_2) \sgn (n_2) n_1 (n_1+2n_2)  q^{Q_B(\bm n)} .
\end{equation*}
\\
The one-dimensional pieces are
\begin{multline*}
-\frac{9}{8} q^{-\frac{1}{24}} \sum_{n_1 \in \Z} (-1)^{n_1} \sgn(n_1) q^{\frac{3}{2} \left(n_1+\frac{1}{6}\right)^2} + \frac 18q^{-\frac 38}\sum_{n_1 \in \Z}(-1)^{n_1}\sgn\left(n_1+\frac 12\right)q^{\frac 32\left(n_1+\frac 12\right)^2}
\\+\frac{9}{2} q^{-\frac{1}{24}} \sum_{n_2 \in \Z} (-1)^{n_2} \sgn(n_2)\left(n_2+\frac{1}{6}\right)^2 q^{\frac{3}{2} \left(n_2+\frac{1}{6}\right)^2}
\\
-{\frac 92q^{-\frac 38}\sum_{n_2 \in \Z}(-1)^{n_2}\sgn\left(n_2+\frac 12\right)\left(n_2+\frac 12\right)^2q^{\frac 32\left(n_2+\frac 12\right)^2}}.
\end{multline*}

We next consider the remaining pieces following similar computations. 
\begin{itemize}[leftmargin=*]
\item  
Firstly, we have
\begin{align*}
\sum_{\bm r \in \mathcal{S}_B}& \ve_B(\bm r) C_1(\bm r)\\
&=	\frac{9}{4} q^{-\frac{5}{12}}
\lp \sum_{\bm n\in\Z^2+\left(\frac{1}{3},\frac{1}{6}\right)}
- \sum_{\bm n\in\Z^2+\left(\frac{1}{3},\frac{1}{2}\right)} \rp 
(-1)^{n_1-\frac 13} (1+\sgn(n_1)\sgn(n_2)) (n_1+2n_2)^2q^{Q_B(\bm n)}
\\
 & \hspace*{+7pt} -\frac{9}{4} \sum_{n_2\in\Z} \sgn(n_2) \lp 2n_2+\frac{2}{3} \rp^2 q^{3n_2^2+2n_2}
+\frac{9}{4} \sum_{n_2\in\Z} \sgn(n_2) \lp 2n_2+\frac{1}{3} \rp^2  q^{3n_2^2+n_2}
\\
& \hspace*{7pt} -\frac{9}{2} \sum_{n_1= 0}^\infty (-1)^{n_1} \left(n_1+\frac{1}{3}\right)^2 q^{\frac{3n_1^2}{2}+\frac{3n_1}{2}}
-\frac{9}{4} \sum_{n_1= 0}^\infty (-1)^{n_1} \left(n_1-\frac{2}{3}\right)^2 q^{\frac{3n_1^2}{2}-\frac{n_1}{2}} 
\\
&\hspace{200pt}+\frac{9}{4} \sum_{n_1 = 0}^\infty (-1)^{n_1} \left(n_1-\frac{1}{3}\right)^2 q^{\frac{3n_1^2}{2}+\frac{n_1}{2}},
\end{align*}
\end{itemize}
\begin{align*}
\sum_{\bm r \in \mathcal{S}_B} &\ve_B(\bm r) C_2(\bm r) \\
&=	
\frac 94 q^{-\frac{5}{12}}
\lp \sum_{\bm n\in\Z^2+\left(\frac 13,\frac 12\right)}
- \sum_{\bm n\in\Z^2+\left(\frac 13,\frac 16\right)} \rp  (-1)^{n_1-\frac 13}(1-\sgn(n_1)\sgn(n_2))(n_1+2n_2)^2q^{Q_B(\bm n)}
\\
&
\hspace*{+7pt} +\frac{9}{4} \sum_{n_2\in\Z} \sgn(n_2)\lp 2n_2+\frac{1}{3}\rp^2 q^{3n_2^2+n_2}-\frac{9}{4} \sum_{n_2\in\Z} \sgn(n_2)\lp 2n_2+\frac{2}{3}\rp^2 q^{3n_2^2+2n_2} 
\\
& \hspace*{+7pt}-\frac{9}{4}\sum_{n_1 = 1}^\infty(-1)^{n_1}\lp n_1+\frac{1}{3}\rp^2 q^{\frac{3n_1^2}{2}-\frac{n_1}{2}}
+\frac{9}{4}\sum_{n_1 = 1}^\infty(-1)^{n_1}\lp n_1+\frac{2}{3}\rp^2 q^{\frac{3n_1^2}{2}+\frac{n_1}{2}}
\\&\hspace{200pt}-\frac{9}{2}  \sum_{n_1 = 0}^\infty (-1)^{n_1} \lp n_1+\frac{2}{3} \rp^2 q^{\frac{3n_1^2}{2}+\frac{3n_1}{2}}.
\end{align*}
Combining the contributions from $C_1, C_2$, and $C_3$, we then find that
\begin{multline*}
\sum_{\bm r \in \mathcal{S}_B}\ve_B(\bm r) \lp C_1(\bm r) + C_2(\bm r) + C_3(\bm r)  \rp
\\ =
\frac 92q^{-\frac{5}{12}}\sum_{\bm n\in\Z^2+\left(\frac 13,\frac 16\right)} 
(-1)^{n_1-\frac 13}  \sgn(n_1) \lp \sgn(n_2) + \sgn (n_1 + n_2) \rp (n_1+2n_2)^2 q^{Q_B(\bm n)}
\\+
\frac 92q^{-\frac{5}{12}}\sum_{\bm n\in\Z^2+\left(\frac 13,\frac 16\right)} 
(-1)^{n_1-\frac 13} \sgn (n_1+n_2) \sgn (n_2) n_1(n_1+2n_2) q^{Q_B(\bm n)}
\\
+\frac{1}{2} 
+18 q^{-\frac{1}{12}} \sum_{n \in \Z+\frac{1}{6}} \sgn (n) n^2 q^{3 n^2}
- 18 q^{-\frac{1}{3}} \sum_{n \in \Z+\frac{1}{3}} \sgn (n)  n^2 q^{3 n^2}
\\
+ 9 q^{-\frac{1}{24}} \sum_{n \in \Z+\frac{1}{6}} (-1)^{n-\frac{1}{6}} \sgn(n) n^2 q^{\frac{3n^2}{2}}
-9 q^{-\frac 38}\sum_{n \in \Z+\frac{1}{2}}(-1)^{n-\frac{1}{2}} \sgn(n) n^2 q^{\frac{3n^2}{2}}.
\end{multline*}

\begin{itemize}[leftmargin=*]
\item Next, we have
\begin{multline*}
\sum_{\bm r \in \mathcal{S}_B}\ve_B(\bm r)C_4(\bm r)
= \sum_{\bm r \in \mathcal{S}_B}\ve_B(\bm r) C_5(\bm r) \\
\hspace{2cm}=-\frac{E_2}{8} q^{-\frac{5}{12}}\sum_{\bm n\in \Z^2+\left(\frac 13,\frac 16\right)}(-1)^{n_1-\frac 13}
\sgn(n_1) \lp \sgn(n_2) + \sgn (n_1+n_2) \rp
q^{Q_B(\bm n)}
\\
\hspace{1.9cm}+\frac{E_2}{8}
-\frac{E_2}{8} q^{-\frac{1}{12}} \sum_{n\in \Z+\frac{1}{6}} \sgn(n) q^{3n^2}
+\frac{E_2}{8} q^{-\frac{1}{3}} \sum_{n\in \Z +\frac{1}{3}} \sgn(n) q^{3n^2}
\\-\frac{E_2}{8} q^{-\frac{1}{24}} \sum_{n\in \Z+\frac{1}{6}} (-1)^{n-\frac{1}{6}} \sgn(n) q^{\frac{3n^2}{2}}
+\frac{E_2}{8} q^{-\frac{3}{8}} \sum_{n\in\Z +\frac{1}{2}} (-1)^{n-\frac{1}{2}} \sgn(n) q^{\frac{3n^2}{2}}
\end{multline*}
so that we get 
\begin{equation*}
\sum_{\bm r \in \mathcal{S}_B}\ve_B(\bm r) \lp C_4(\bm r) + C_5(\bm r) +C_6(\bm r)  \rp
= 2 \sum_{\bm r \in \mathcal{S}_B}\ve_B(\bm r) C_4(\bm r) .
\end{equation*}

\item We next consider the contributions from $C_{7,(0,1)}$ and $C_{8,(0,1)}$ to find that
\begin{align*}
-\frac{2\vt\left(\frac 12\right)^2}{\eta^6}\sum_{\bm r\in\mathcal{S}_B} \ve_B(\bm r)C_{7,(0,1)}(\bm r)&=-q^{-\frac{5}{12}}\sum_{\bm n\in \Z^2+\left(\frac 13,\frac 16\right)}
\sgn(n_1) \lp \sgn(n_2) + \sgn (n_1+n_2) \rp
q^{Q_B(\bm n)}\\&\hspace{-120pt}
-1+
q^{-\frac{1}{12}} \sum_{n\in \Z +\frac{1}{6}} \sgn(n) q^{3n^2}
+q^{-\frac{1}{3}} \sum_{n\in\Z+\frac{1}{3}} \sgn(n) q^{3n^2}
+q^{-\frac{1}{24}} \sum_{n\in\Z+\frac{1}{6}} q^{\frac{3n^2}{2}} 
-q^{-\frac{3}{8}} \sum_{n\in\Z+\frac{1}{2}} q^{\frac{3n^2}{2}} ,
\\
- \frac{2\vartheta \lp \frac{1}{2}\rp^2}{\eta^6} \sum_{\bm r \in \mathcal{S}_B} \ve_B(\bm r) C_{8,(0,1)}(\bm r) 
&=-q^{-\frac{5}{12}}\sum_{\bm n\in \Z^2+\left(\frac 13,\frac 16\right)}
\sgn(n_1) \lp \sgn(n_2) + \sgn (n_1+n_2) \rp
q^{Q_B(\bm n)}
\\
&\hspace{-120pt}
-1
+q^{-\frac{1}{12}} \sum_{n\in \Z +\frac{1}{6}} \sgn(n) q^{3n^2}
+q^{-\frac{1}{3}} \sum_{n\in\Z+\frac{1}{3}} \sgn(n) q^{3n^2}
+q^{-\frac{3}{8}} \sum_{n\in\Z+\frac{1}{2}} q^{\frac{3n^2}{2}} 
-q^{-\frac{1}{24}} \sum_{n\in\Z+\frac{1}{6}} q^{\frac{3n^2}{2}} .
\end{align*}
They combine as
\begin{multline*}
\frac{\vartheta \lp \frac{1}{2}\rp^2}{\eta^6}  
\sum_{\bm r \in \mathcal{S}_B} \ve_B(\bm r) \lp C_{7,(0,1)}(\bm r) + C_{8,(0,1)}(\bm r) \rp
\\
=
q^{-\frac{5}{12}}\sum_{\bm n\in \Z^2+\left(\frac 13,\frac 16\right)}
\sgn(n_1) \lp \sgn(n_2) + \sgn (n_1+n_2) \rp
q^{Q_B(\bm n)}
\\
+1
-q^{-\frac{1}{12}} \sum_{n\in \Z +\frac{1}{6}} \sgn(n) q^{3n^2}
-q^{-\frac{1}{3}} \sum_{n\in\Z+\frac{1}{3}} \sgn(n) q^{3n^2}.
\end{multline*}

\item Finally, we consider the case $\bm\ell=(1,\ell_2)$ ($\ell_2\in\{0,1\}$) and determine
\begin{align*}
\frac{2\vartheta \left(\frac{\tau+\ell_2}{2}\right)^2 }{\eta^6} & \sum_{\bm r \in \mathcal{S}_B}\ve_B(\bm r) C_{7,(1,\ell_2)}(\bm r) 
=
\frac{2\vartheta \left(\frac{\tau+\ell_2}{2}\right)^2 }{\eta^6} \sum_{\bm r \in \mathcal{S}_B}\ve_B(\bm r) C_{8,(1,\ell_2)}(\bm r) 
\\
&=
(-1)^{\ell_2} q^{-\frac{2}{3}} \sum_{\bm n\in\Z^2+\left(\frac{1}{3},\frac{2}{3}\right)} (-1)^{(\ell_2+1)\lp n_1- \frac{1}{3} \rp} 
\sgn(n_1) \lp \sgn(n_2) + \sgn (n_1+n_2) \rp q^{Q_B(\bm n)}
\\
& \hspace{70pt} - q^{-\frac{1}{3}} \sum_{n\in\Z+\frac{1}{3}} \sgn (n) q^{3n^2}
+(-1)^{\ell_2} q^{-\frac{7}{12}} \sum_{n\in\Z+\frac{1}{6}} \sgn (n) q^{3n^2}.
\end{align*}
\end{itemize}

The claim of the theorem now follows by a direct calculation using the following sign-identity on the two-dimensional contributions from $C_7$ and $C_8$ together with the change of variables $\bm n  \mapsto (-n_1-2n_2,n_2)$ which leaves both $Q_B$ and the respective lattice shifts invariant:
\begin{equation*}
\sgn(n_1)(\sgn(n_2)+\sgn(n_1+n_2))-\sgn(n_1+2n_2)(\sgn(n_2)-\sgn(n_1+n_2))=2\sgn(n_1+2n_2)\sgn(n_1).  \qedhere
\end{equation*}

\end{proof}

\subsection{Modular properties of the vacuum character}\hfill

As in the case of parafermionic characters of type $A_2$, we focus on the rank two contributions $\Phi (\t)$ and $\LL_{\bm a} (\t)$. For $\LL_{\bm a} (\t)$, the two vectors determining the factors inside the sign functions are orthogonal with respect to the quadratic form $Q_B$. Therefore, $\LL_{\bm a} (\t)$ can be written in terms of products of two rank one false theta functions. This leaves us with $\Phi (\t)$ as the only nontrivially rank two piece in the decomposition, which we study next. We start by defining
\begin{align*}
f_0 (\bm w) 
&:= \vth_{3,1}^{[1]} (w_1) \vth_{3,2}^{[1]} (w_2) - \vth_{3,2}^{[1]} (w_1) \vth_{3,1}^{[1]} (w_2),
\quad\hspace{0.25cm}
f_1 (\bm w) 
:= \vth_{3,1}^{[1]} (w_1) \vth_{3,2}^{[3]} (w_2) - \vth_{3,2}^{[1]} (w_1) \vth_{3,1}^{[3]} (w_2),
\\
g_0 (\bm w) 
&:=\vth_{\frac{3}{2},1}^{[1]} (w_1) \vth_{\frac{3}{2},0}^{[1]} (w_2)  
-\vth_{\frac{3}{2},0}^{[1]} (w_1)  \vth_{\frac{3}{2},1}^{[1]} (w_2),
\quad
g_1 (\bm w) 
:=\vth_{\frac{3}{2},1}^{[1]} (w_1) \vth_{\frac{3}{2},0}^{[3]} (w_2)  
-\vth_{\frac{3}{2},0}^{[1]} (w_1)  \vth_{\frac{3}{2},1}^{[3]} (w_2).
\end{align*}
A direct calculation, using Lemma \ref{SignLemma} then gives:
\begin{lem}\label{lem:F1int}
We have 
\begin{multline*}
\Phi(\t) =
\frac 23 \int_\t^{\t +i\infty} \int_\t^{w_1} 
\frac{72f_1(\bm w) -E_2(\t) f_0(\bm w) }{\sqrt{i(w_1-\t)} \sqrt{i(w_2-\t)}} dw_2 dw_1
\\+
\frac 16 \int_\t^{\t +i\infty} \int_\t^{w_1} 
\frac{36g_1(\bm w) -E_2(\t) g_0(\bm w) }{\sqrt{i(w_1-\t)} \sqrt{i(w_2-\t)}} dw_2 dw_1.
\end{multline*}
\end{lem}	
Using integration by parts, while noting that $\vth_{m,r}^{[3]} (\t) = \frac{1}{2\pi im} \frac{\del}{\del \t} \vth_{m,r}^{[1]} (\t)$
and $f_0(w_1,w_1) = g_0 (w_1,w_1) =0$, we obtain the following:
\begin{prop}\label{prop_phi_integral_form}
We have
\begin{align*}
\Phi(\tau) = 
\frac{1}{\pi} \int_\tau^{\tau +i\infty} \rint{\tau}{w_1} 
\frac{\lp 4f_0(\bm w)+g_0(\bm w)\rp \lp 1-\frac{\pi i}6(w_2-\tau)E_2(\tau)\rp}
{\sqrt{i(w_1-\tau)} (i(w_2-\tau))^{\frac 32}} dw_2  dw_1 .
\end{align*}
\end{prop}
In parallel with Section 5.2, we define the completion of $\Phi$ as
\begin{align*}
\widehat{\Phi}(\t,w):= 
\frac{1}{\pi} \int_\t^w \rint{\t}{w_1} 
\frac{\lp 4f_0(\bm w)+g_0(\bm w)\rp \lp 1-\frac{\pi i}6(w_2-\t)E_2(\t)\rp}
{\sqrt{i(w_1-\t)} (i(w_2-\t))^{\frac32}} dw_2  dw_1 .
\end{align*}
 The following proposition shows the transformation law of $\wh\Phi$.  
\begin{prop}
\label{prop:Phi_modular_properties}
For $M = \left(\begin{smallmatrix}
a & b\\ c & d
\end{smallmatrix}\right)\in \operatorname{SL}_2(\Z)$, we have
\begin{equation*}
\wh{\Phi}  \lp \frac{a \t + b}{c \t + d}, \frac{a w + b}{c w + d} \rp
=
\nu_{\eta} (M)^{10} (c\t +d)^3  \, \wh{\Phi}(\t,w).
\end{equation*}
\end{prop}
\begin{proof}
It is enough to verify the claim for translation and inversion, in which case it reads
\begin{align*}
\wh{\Phi} (\t+1, w+1)=	e^{\frac{5\pi i}{6}}\wh{\Phi} (\t, w),\qquad 
\wh{\Phi}\left(-\frac{1}{\t},-\frac{1}{w}\right)=-i\t^3 \wh{\Phi} (\t, w).
\end{align*}
As in Proposition \ref{prop:Psi_modular_properties}, these transformations follow from the following modular transformations for $f_0$ and $g_0$:
\begin{align*}
&f_0(\bm w+(1,1))= e^{\frac{5\pi i}{6}} f_0(\bm{w}), 
\hspace{1cm}
f_0\left(-\frac{1}{w_1},-\frac{1}{w_2}\right)=-i w_1^\frac32 w_2^\frac32f_0(\bm{w}),\\
& g_0(\bm w+(1,1))= e^{\frac{5\pi i}{6}} g_0(\bm{w}), \hspace{1cm}g_0\left(-\frac{1}{w_1},-\frac{1}{w_2}\right)=-i w_1^\frac32 w_2^\frac32g_0(\bm{w}). \qedhere
\end{align*}
\end{proof}

\section{Higher rank false theta functions from Schur indices and $\hat{Z}$-invariants}\label{sec:Schur_Zhat}

In this section, we study further examples of rank two false theta functions coming from Schur's indices \cite{BNi,Creutzig0} and $\hat{Z}$-invariants \cite{Gukov}.

\subsection{False theta functions from Schur indices}

A remarkable correspondence (or duality) between four-dimensional $\mathcal{N} = 2$ 
superconformal field theories (SCFTs) and vertex operator algebras was recently found in \cite{Beem}.
According to \cite{Beem}, the Schur index of a $\mathcal{N}=2$ SCFT agrees with the character of 
a vertex operator algebra. As mentioned above, the Schur index of the $(A_1,D_{2k+2})$ Argyres--Douglas theory is a meromorphic Jacobi form of negative index in two variables $\bm z=(z_1,z_2)$. It was demonstrated in \cite{BNi,Creutzig0} that  the index agrees with the character of a certain affine $W$-algebra. Up to some Euler factors and change of variables, this character is given by the Jacobi form $\mathbb{T}_k(\bm z;\tau)$ defined in Section \ref{sec:preliminaries} (see \cite{Creutzig0}).  
Here we analyze its Fourier coefficients. For $k \in \mathbb{N}$, letting
\begin{equation*}
\mathbb F_k(\t):=\frac{1}{2} \sum_{\bm n\in\Z^2+\left(0,\frac 12\right)} (-1)^{n_1}\sgn(n_1)\sgn\left(n_2\right)
q^{\frac{n_1^2}{2}+n_1n_2+(k+1)n_2^2},
\end{equation*}
it is not hard to prove the following result using slight adjustments of \cite[Lemma 3.5]{BKMZ}. 
\begin{prop} \label{prop:Tk_constant_term}
For $ |q| < |\zeta_1|,  |\zeta_2|, |\zeta_1 \zeta_2|<1$ and $\bm{r} \in \mathbb{Z}^2$, the $\bm{r}$-th Fourier coefficient of $\frac{\eta(\tau)^3 \eta(\frac{k+1}{2} \tau )^2}{\eta((k+1)\tau)} \mathbb T_k(\bm z;\tau)$ equals
\begin{align*}
q^{\frac{k+1}{4} (2r_2+1)}
\sum_{\bm{n} \in \mathbb{Z}^2} (-1)^{n_1} \varrho_{n_1,n_2+r_1}\varrho_{n_2+r_2,n_2} q^{\frac{n_1(n_1+1)}{2}+n_1 (n_2+r_1)+(k+1)n_2^2+(k+1)(r_2+1)n_2},
\end{align*}
where $\varrho_{m,n}:=\frac{1}{2}( {\rm sgn} ( m+ \frac{1}{2} )+{\rm sgn} ( n+ \frac{1}{2} ) )$.
In particular,
$$\mathbb F_k(\t)=\frac{\eta(\tau)^3 \eta\left(\frac{k+1}{2} \tau\right)^2}{\eta((k+1)\tau)} 
{\rm CT}_{[\bm \zeta]}(\mathbb T_k(\bm z;\tau))  .$$
\end{prop}
\begin{proof}
We first let
$$h({\bm z};\tau):=-\frac{i \eta((k+1)\tau)^3 \zeta_2^{-1} q^{-\frac{k+1}{2}} \vartheta(z_1;(k+1)\tau)}{\vartheta\left(z_2+\frac{k+1}{2} \tau; (k+1) \tau\right)\vartheta\left(z_1+z_2+\frac{k+1}{2} \tau; (k+1)  \tau\right)}.$$
As in the proof of \cite[Lemma 3.5]{BKMZ}, we obtain an expansion
\begin{equation*}
h({\bm z};\tau)=\sum_{(n_3,n_4) \in \mathbb{Z}^2} \varrho_{n_3,n_4} q^{(k+1)n_3 n_4+\frac{k+1}{2} n_3+\frac{k+1}{2} n_4} \zeta_1^{-n_4} \zeta_2^{n_3-n_4}.
\end{equation*}
This combined with the well-known formula (see \cite[formula (2.1)]{Andrews})
$$-\frac{i \zeta_1^{-\frac12} \eta(\tau)^3}{\vartheta(z_1;\tau)}=\sum_{\bm n \in \mathbb{Z}^2} \varrho_{n_1,n_2}(-1)^{n_1}q^{\frac{n_1(n_1+1)}{2}+n_1 n_2} \zeta_1^{n_1},$$
easily implies the statement. \qedhere
\end{proof}

A direct calculation using Lemma \ref{SignLemma} then shows the following. 
\begin{prop}\label{prop:Fk_iterated_integral} We have 
\begin{align*}
\frac{2 \mathbb F_k(\t) }{\sqrt{2k+1}}  
=& \int_{\tau}^{\tau+i\infty} 
\frac{\eta((2k+1)w_1)^3}{\sqrt{i(w_1-\tau)}} \int_{\tau}^{w_1}
\frac{\eta(w_2)^3}{\sqrt{i(w_2-\tau)}}dw_2dw_1 \\
&+2(k+1) \sum_{j=0}^{2k+1}(-1)^{j} 
\int_{\tau}^{\tau+ i \infty} \frac{ \vartheta^{[1]}_{k+1,j}((2k+1)w_1)}{\sqrt{i(w_1-\tau)}}\int_{\tau}^{w_1} \frac{\vartheta^{[1]}_{k+1,j+k+1}(w_2)}{\sqrt{i(w_2-\tau)}}dw_2dw_1.
\end{align*}
\end{prop}

We now specialize to $k=1$. This recovers the $A_2$ false theta function entering the character formula of the $W^0(2)_{A_2}$ vertex algebra studied in \cite{AMW,BKM,BKMZ}. In this case, the right-hand side in Proposition \ref{prop:Fk_iterated_integral} simplifies and we obtain an elegant integral representation
\begin{align*} \label{Eichlerk1}
\mathbb{F}_1(\tau)=\frac{3\sqrt{3}}{4} \int_{\tau}^{\tau+ i \infty} \frac{ \eta(3w_1)^3}{\sqrt{i(w_1-\tau)}}\int_{\tau}^{w_1} \frac{\eta(w_2)^3}{\sqrt{i(w_2-\tau)}}dw_2dw_1,
\end{align*} 
as a consequence of the identity
$$\sum_{j=0}^3 (-1)^j \vartheta^{[1]}_{2,j}(3w_1)\vartheta^{[1]}_{2,j+2}(w_2)=\frac{1}{8} \eta(3 w_1)^3 \eta(w_2)^3.$$
This integral admits a modular completion $\widehat{\mathbb{F}}_1(\tau, w)$ (see Section \ref{sec:A_2_parafermion}) whose modular 
transformation properties under $\mathrm{SL}_2(\mathbb{Z})$ can be easily analyzed.

\subsection{$\hat{Z}$-invariants of $3$-manifolds from unimodular ${\tt H}$-graphs} 
The methods of this paper can also be used in analyzing the modular properties of $\hat{Z}$-invariants or homological blocks of $3$-manifolds. These are certain $q$-series with integer coefficients proposed by \cite{Gukov} as a new class of 3-manifold invariants. Remarkably, these $q$-series, which are convergent on the unit disk, are designed and expected to produce the WRT (Witten--Reshetikhin--Turaev) invariants of the relevant manifolds through the radial limits of the parameter $q$ to the roots of unity.

More concretely, we restrict our attention to plumbed $3$-manifolds whose plumbing graphs are trees. The vertices of the plumbing graph, which we label by $\{ v_j\}_{1 \leq j \leq N}$, are decorated with a set of integers $m_{jj}$ for $1 \leq j \leq N$. This data then determines the linking matrix $M=(m_{jk})_{1 \leq j,k \leq N}$ by setting the off-diagonal entries $m_{jk}$ to $-1$ if the associated vertices $v_j$ and $v_k$ are connected by an edge in the graph and by setting it to $0$ otherwise.\footnote{
Here we follow the conventions of \cite{BMM} and switch the sign of the linking matrix $M$ compared to \cite{Gukov}. 
}
We further restrict to cases in which the matrix $M$ is positive definite. 
Finally, we define the shift vector $\bm{\d} := (\d_j)_{1 \leq j \leq N}$ where $\d_j \equiv \mathrm{deg} (v_j) \pmod{2}$ and $\mathrm{deg} (v_j)$ denotes the degree of the vertex $v_j$. Then the $\hat{Z}$-invariant is defined for each equivalence class  $\bm{a} \in 2 \mathrm{coker}(M) + \bm{\d}$ by
\begin{equation*}
\hat{Z}_{\bm a}(q) := \frac{q^{\frac{-3N+{\mathrm{tr}}(M)}{4}}}{(2\pi i)^N} \, \mathrm{P.V.} 
\int_{|w_1|=1}\ldots\int_{|w_N|=1}
\prod_{j=1}^N \big(w_j - w_j^{-1}\big)^{2-\mathrm{deg}(v_j)} \,
\Theta_{-M, \bm a}(q;{\bm w}) \frac{dw_N}{w_N} \ldots \frac{dw_1}{w_1},
\end{equation*}
where 
\begin{equation*}
\Theta_{-M, \bm a}(q;{\bm w}):=\sum_{\boldsymbol{\ell} \in 2 M \mathbb{Z}^N+\bm a } q^{\frac14\boldsymbol{\ell} ^T M^{-1} \boldsymbol{\ell} } {\bm w}^{{\bm \ell}}
\end{equation*}
and the integrals are defined using the Cauchy principal value (as indicated by the notation $\mathrm{P.V.}$) and performed in counterclockwise direction. If more specifically the linking matrix is invertible (unimodular), in which case we also call the associated plumbing graph unimodular, then $\mathrm{coker} (M) = 0$ and there is only one $\hat{Z}$-invariant.

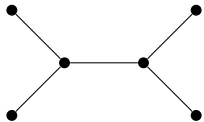
\begin{figure}[h]
\centering
	\begin{tikzpicture}[scale=0.7]
		\node[shape=circle,fill=black, scale = 0.4] (1) at (0,0) { };
		\node[shape=circle,fill=black, scale = 0.4] (2) at (-1,1) { };
		\node[shape=circle,fill=black, scale = 0.4] (3) at (-1,-1)  { };
		\node[shape=circle,fill=black, scale = 0.4] (4) at (1.5,0) { };
		\node[shape=circle,fill=black, scale = 0.4] (5) at (2.5,-1) { };
		\node[shape=circle,fill=black, scale = 0.4] (6) at (2.5,1) { };	
		\path [-] (1) edge node[left] {} (2);
		\path [-](1) edge node[left] {} (3);
		\path [-](1) edge node[left] {} (4);
		\path [-](4) edge node[left] {} (5);
		\path [-](4) edge node[left] {} (6);	
	\end{tikzpicture}
	\caption{The {\tt H}-graph} \label{fig:H_graph}
\end{figure}

\noindent
The $\hat{Z}$-invariants are conjectured to yield quantum modular forms, which for example can be verified in the case of unimodular, $3$-star plumbing graphs for which the relevant invariants can be written in terms of unary false theta functions \cite[Proposition 4.8]{GM} (see also \cite{BMM,CCFGH}) The simplest plumbing graph for which the corresponding homological block can not be written in terms of one-dimensional false theta functions is the ${\tt H}$-graph (see Figure \ref{fig:H_graph}). 
For unimodular ${\tt H}$-graph, two of the authors and Mahlburg computed the $\hat{Z}$-invariants and studied their higher depth quantum modular properties in \cite{BMM}. 
We explain how the method of Section \ref{sec:sign_lemma} can be used to analyze their modular properties. 
As demonstrated in 
\cite{BMM}, the relevant $\hat{Z}$-invariants can be expressed as a difference of two series of the form
\begin{equation*} \label{FS}
F_{\mathcal S,Q,\varepsilon}(\tau) :=\sum_{\bm \alpha\in\mathcal S}\varepsilon({\bm \alpha}) \sum_{{\bm n} \in \mathbb{N}_0^2} q^{KQ({\bm n}+{\bm \alpha})}.
\end{equation*}
Here $Q({\bm n})=:\sigma_1n_1^2+2\sigma_2n_1n_2+\sigma_3n_2^2$ with $\s_1,\s_2,\s_3 \in \IZ$ defines a positive definite quadratic form and $\mathcal S\subset \mathbb{Q}_{>0}^2$ is a finite set with the property that $(1,1)-{\bm \alpha}$, $(1-\alpha_1,\alpha_2)\in\mathcal S$ for ${\bm \alpha}\in\mathcal S$, $\varepsilon({\bm \alpha})=\varepsilon((1,1)-{\bm \alpha})=\varepsilon((1-\alpha_1,\alpha_2))$, and $K\in \mathbb{N}$ is minimal such that $\mathcal A:= K\mathcal S\subset \N^2$. For explicit formulas for $Q$ and $\mathcal{S}$ see \cite{BMM}. 
We can use the symmetry in the sum over $\bm \alpha$ to obtain that 
\begin{equation*}
F_{\mathcal S,Q,\varepsilon}(\tau) = \frac14 \sum_{{\bm \alpha} \in\mathcal S} \varepsilon({\bm \alpha}) \sum_{{\bm n} \in\Z^2+{\bm \alpha}} \sgn(n_1) \lp \sgn(n_1)+\sgn(n_2) \rp q^{KQ(\bm n)}.
\end{equation*}
The contribution from $\sgn(n_1)\sgn(n_1)=1$ yields a theta function which is a modular form. For the contribution from $\sgn(n_1)\sgn(n_2)$ we proceed as in Section \ref{sec:sign_lemma} to obtain a representation of $\hat{Z}$ in terms of
double integrals and ordinary theta functions.
\begin{prop}\label{lem:F}
We have
\begin{align*}
&F_{\mathcal S,Q,\varepsilon}(\tau) 
\\ & \ = 
\frac{K \s_3 \sqrt{D}}{2}   \sum_{\substack{\bm{\a} \in\mathcal S, \\  r \pmod{\s_3} } }
 \varepsilon({\bm \alpha})
\int_{\tau}^{\tau+i\infty} \frac{  \vth^{[1]}_{KD\s_3,2KD (\a_1+r)} (w_1)  }{ \sqrt{i (w_1-\tau)} }
\int_{\tau}^{w_1} \frac{  \vth^{[1]}_{K\s_3, 2K (\s_2 (\a_1+r) + \s_3 \a_2)} (w_2) }{\sqrt{i (w_2-\tau)}}d w_2 d w_1
\\ &\quad+ 
\frac{K \s_1 \sqrt{D}}{2}  \sum_{\substack{\bm{\a} \in\mathcal S, \\  r \pmod{\s_1} } }
 \varepsilon({\bm \alpha})
\int_{\tau}^{\tau+i\infty} \frac{ \vth^{[1]}_{KD\s_1,2KD (\a_2+r)} (w_1)  }{ \sqrt{i (w_1-\tau)} }
\int_{\tau}^{w_1} \frac{   \vth^{[1]}_{K\s_1, 2K (\s_2 (\a_2+r) + \s_1 \a_1)} (w_2) }{\sqrt{i (w_2-\tau)}}d w_2 d w_1
\\ & \qquad +
\frac{1}{4} \lp 1 - \frac{2}{\pi}  \arctan\left(\frac{\s_2}{\sqrt{D}}\right) \rp
\sum_{\substack{\bm{\a} \in\mathcal S, \\  r \pmod{\s_3} } }  \varepsilon({\bm \alpha})
\vth_{KD\s_3,2KD (\a_1+r)} (\t)  \, \vth_{K\s_3, 2K (\s_2 (\a_1+r) + \s_3 \a_2)} (\t) ,
\end{align*}
where $D:=\sigma_1\sigma_3-\sigma_2^2$.
\end{prop}

\section{Conclusion and future work}\label{sec:conclusion}
In this paper, modular properties of rank two false theta functions are studied following the recent developments in depth two mock modular forms. These results are then used to study characters of parafermionic vertex algebras of type $A_2$ and $B_2$. A natural question is then how our results extend to parafermions associated to other simple Lie algebras. The only remaining rank two simple Lie algebra $G_2$ is a natural setting where our approach would directly apply. A more interesting problem is the extension to higher rank Lie algebras such as $A_3$. The approach we use in Sections \ref{sec:A_2_parafermion} and \ref{sec:B_2_parafermion} to compute the constant term of meromorphic Jacobi forms would still be applicable, albeit becoming computationally more expensive as the number of roots increases. Although being straightforward, computations of the linear combinations that give the characters of parafermionic vertex algebras were a particularly strenuous part of the calculations. Therefore, it would be desirable to streamline this part of the computation ahead of the generalizations.

The modular properties for these higher rank cases, on the other hand, can in principle be studied again following the corresponding structure for mock modular forms. The details on higher depth mock modular forms are developed in \cite{ABMP, FK, Kudla, Naz, Westerholt} and we leave it as future work to form this connection. Another interesting prospect would be to understand and make predictions on these modular behaviors (weights, multiplier systems, etc.) through either physical or algebraic methods.

A slightly different direction would be studying the modular properties of Fourier coefficients of the character of $V_{k}(\text{sl}_3)$ at the boundary admissible levels $k=-3+\frac{3}{j}$, where $j \geq 2$ and $\gcd(j,3)=1$, generalizing the results for $j=2$ obtained in \cite{BKMZ} (see also Section \ref{sec:Schur_Zhat}). This problem essentially requires analyzing the Fourier coefficients of the Jacobi form (see \cite{KW0})
$$\frac{\vartheta(z_1; j \tau)\vartheta(z_2; j \tau)\vartheta(z_1+z_2; j \tau)}{\vartheta(z_1;\tau)\vartheta(z_2;\tau)\vartheta(z_1+z_2;\tau)},$$
which can be handled using the methods of this paper.

\end{document}